\documentclass{amsart}
\usepackage{amssymb,amsfonts,amsmath}
\usepackage{color}
\usepackage{graphicx}

\copyrightinfo{2010}{American Mathematical Society}
\newtheorem{theorem}{Theorem}[section]
\newtheorem{lemma}[theorem]{Lemma}
\newtheorem{corollary}[theorem]{Corollary}
\newtheorem{proposition}[theorem]{Proposition}

\theoremstyle{definition}
\newtheorem{definition}[theorem]{Definition}
\newtheorem{example}[theorem]{Example}
\newtheorem{problem}[theorem]{Problem}

\theoremstyle{remark}

\numberwithin{equation}{section}

\begin{document}

\title[On uniqueness theorems ]
      {On uniqueness theorems for the inverse problem of Electrocardiography in the 
			Sobolev spaces} 
			
			\author[V. Kalinin]{Vitaly Kalinin}
\address{EP Solutions SA, 
Avenue des Sciences 13, 1400 Yverdon-les-Bains, Switzerland}

\email{vitaly.kalinin@ep-solutions.ch}
		
\author[A. Shlapunov]{Alexander Shlapunov}
\address{Siberian Federal University,
         Institute of Mathematics and Computer Science,
         pr. Svobodnyi 79,
         660041 Krasnoyarsk,
         Russia}		

\email{ashlapunov@sfu-kras.ru}

\author[K. Ushenin]{Konstantin Ushenin}
\address{EP Solutions SA, 
	Avenue des Sciences 13, 1400 Yverdon-les-Bains, Switzerland}

\email{contact@ep-solutions.ch}



\begin{abstract}
We consider a mathematical model related to reconstruction of cardiac electrical activity 
from ECG measurements on the body surface. An application of recent developments in solving  
boundary value problems for elliptic  and parabolic equations in Sobolev type  spaces 
allows us to obtain uniqueness theorems for the model. The obtained results can be used as 
a sound basis for creating numerical methods for non-invasive mapping of the heart.
\end{abstract}

\maketitle

\section*{Introduction}
\label{s.0}

The inverse problem of electrocardiography, i.e. the problem of reconstruction of cardiac 
electrical activity from ECG measurements on the body surface is of great value for 
diagnostics and treatment of cardiac arrhythmias \cite{Cluitmans2018}. The inverse 
electrocardiography problem can be considered in various statements. In this paper we focused 
on the inverse problem of reconstruction of cardiac electrical activity inside the myocardium. 

Electrophysiological processes in the myocardium are most often described using the so-called
bidomain model, see \cite{1}, \cite{geselowitz1983}, \cite{2} and elsewhere.

The bidomain model can be presented in the form of two non-linear parabolic partial differential 
equation of reaction-diffusion type or in the form of a linear elliptic equation of the second 
order and a non-linear reaction-diffusion equation. The reaction term of the parabolic 
equations characterizes transmembrane ionic currents through the potential-sensitive ion 
channels. It is described by a set of ordinary non-linear differential equations (also 
referred to as the ionic model) or, in the simplest case, a nonlinear “activation”{} function. 
The reaction term can also include an external electrical current, that allows describing the 
processes of electrical stimulation of the heart. The bidomain model can be complemented with 
the Laplace equation for the electric field potential outside the myocardium, boundary 
conditions on the torso surface, and interface conditions at the boundary of 
the myocardial domain (the bidomain-bath model).  

The bidomain model was initially proposed in the late 60-70s \cite{Schmidt1969}, 
\cite{Tung1978}, \cite{Miller1983}. Formal derivations of the model equations and the 
boundary conditions with different levels of mathematical rigor were obtained later 
\cite{Neu1993}, \cite{Krassowska1994}, \cite{Pennacchio2005}, \cite{Hand2009}, 
\cite{Collin2018}, \cite{Grandelius2019}, \cite{Bendahmane2019}. The bidomain model is widely 
used for simulation the propagation of the myocardial excitation, which consists in the numerical 
solving the initial-boundary value problem for the 
corresponding  equations \cite{Clayton2010}, \cite{Trayanova2011}, \cite{Beheshti2016}, 
\cite{Qarteroni2017}, \cite{Potse2018}. This initial-boundary value problem was extensively 
studied theoretically. Positive results on the existence, uniqueness, and regularity of the 
"weak"{} and "strong"{} solution to this problem for several versions of the "ionic model"{} in
the framework of the suitable functional spaces were obtained in \cite{Colli-Franzone2002}, 
\cite{Bendahmane2006}, \cite{Bourgault2009}, \cite{Veneroni2009}, \cite{Kunisch2_2013}, 
\cite{Kunisch3_2013}, \cite{Pargaei2019}.

The bidomain model is widely accepted as an accurate model for the cardiac electrical activity 
\cite{1}. Therefore, it seems reasonable to formulate the problem of reconstructing the 
electrical activity inside the myocardium as an inverse problem for the bidomain equations. 

This problem can be attributed to the class of so-called interface problems for partial 
differential equations. However, such problem are expected to be essentially ill-posed unlike 
the classical well-posed transmission problems in the theory of elliptic boundary value
problems.

In most works, only the linear elliptic equations of the bidomain--bath model were used for 
formulation such inverse problem \cite{Nielsen2007}. Inverse problems of this class were 
investigated numerically in a series of works, see, for instance {\cite{He2000}, 
\cite{Liu2006}, \cite{Wang2010}, \cite{Yu2015}, in which various constraints were imposed on 
the solution in order to obtain the uniqueness of the numerical solution and the stability of 
the computational procedure. The authors were able to demonstrate the feasibility of a 
plausible-looking reconstruction of electrical activity inside the myocardium. However, from 
the applied point of view, the physiological adequacy of the solution obtained by this method 
strongly depends on the accuracy of the a priori approximation of the solution.

In contrast to the initial-boundary value problem, the inverse problems for the bidomain mode 
are not sufficiently studied. Recently, theoretical investigations of the inverse problem led 
to interesting results about non-uniqueness and existence of its solutions in Hardy type 
spaces, were presented in see \cite{2K}. However, the results were obtained under the 
following very restrictive assumptions: all the elliptic operators involved in the model 
should be proportional.

Previously, Burger at al. \cite{1} considered the possibility of solving the inverse problem 
for the steady (elliptic) part of the bidomain model using a priori information about the 
desired solution. More precisely, they formulated the inverse problem (in terms of the 
transmembrane potential) as a problem of minimizing the norm of the difference between the 
desired and a priori solution, provided that the solution satisfies the elliptic equations 
and the boundary conditions. The authors proved the uniqueness theorem for solving the 
inverse problem in this statement. This result is considered as a theoretical justification 
for the above-mentioned numerical methods. 

The inverse problem for the complete bidomain model in the form of two reaction-diffusion 
equations were studied in \cite{Ainseba2015}. This inverse problem provides a reconstruction 
of the electrical activity inside the myocardium in a special case, when the heart is 
activated by electrical stimulation subject to known initial conditions, for example, under 
the initial conditions which the myocardium has at rest. The inverse problem was stated as a 
problem of identification of the electrical stimulation current by known electrical potential 
on the body surface in the form of an optimal control problem. Using the simple two-variable 
Mitchell-Schaeffer ionic model, authors obtained results on the existence of the solution to 
this problem. 

The present article is devoted to the study of the uniqueness of the solution to the problem 
of reconstruction of the electrical activity of the heart inside the myocardium for the 
bidomain model without strict restrictions (in the physiological sense) on the  cardiac activation patterns.
We aimed to describe conditions, providing uniqueness the theorem for both 
steady and evolutionary versions of the bidomain model in an essentially general situation 
involving both elliptic and parabolic differential operators. In this study we used a 
simplified linear version of the reaction term of the parabolic equation of the bidomain 
model. The linear assumption on the reaction part was utilized in recent papers 
\cite{Aniseba2_2015}, \cite{Wu2018} for studying the inverse problem of reconstruction of 
electrical conductivities for the bidomain model by the Carleman estimate technique. 
Despite the fact that the linear activation function is a significant simplification from a 
physiological point of view, the results of the analysis of the linear version of the 
bidomain model can serve as a starting point for the study of more complex and 
physiologically adequate models.

In our investigation we use 
developments related to the ill-posed Cauchy problem for elliptic 
equations, see \cite{KMF91},  \cite{Lv1}, \cite{Tark36}, to the Dirichlet problem and the Neumann 
problem for strongly elliptic operators possessing the Fredholm property, see, for instance, 
\cite{GiTru83}, \cite{McL00}, \cite{Mikh}, \cite{Roit96}, \cite{Simanca1987}, 
and to the non-standard Cauchy problem for parabolic equations, see \cite{KuSh}, 
\cite{PuSh}, regarding the bidomain model as a transmission problem, see, 
for instance, \cite{Bor10} (cf. also \cite{Shef} for more general models).
The results of sections \S \ref{s.bidomain}, 
\S \ref{s.bidiomain.t} belong to V. Kalinin and A. Shlapunov, the numerical 
part in \S \ref{s.numerical} is due to V. Kalinin and K. Ushenin. 

\section{Mathematical preliminaries}
\label{s.math.prelim}

Let $\theta$ be a measurable set in ${\mathbb R}^n$, $n\geq 2$.
Denote by $L^2(\theta)$ a Lebesgue space of functions on  $\theta$ with the 
standard inner product $(\cdot, \cdot)_{L^2 (D)}$.
If $D$ is a domain in ${\mathbb R}^n$ with a piecewise smooth boundary
$\partial D$, then for $s \in \mathbb{N}$ we denote by $H^s(D)$ 
the standard Sobolev space with the standard inner product $( \cdot, \cdot )_{H^s(D)}$.
It is well-known that this scale extends for all $s>0$. More precisely, 
given any non-integer $s \in  {\mathbb R}_+ \setminus {\mathbb Z}_+ $, we use 
the so-called Sobolev-Slobodetskii space $H^s (D)$, 
see \cite{Slob58}.

Denote by $H^s_0 (D)$ the closure of the subspace $C^{\infty}_{\mathrm{comp}} (D)$ in 
$H^{s} (D)$, where $C^{\infty}_{\mathrm{comp}} (D)$ is the linear space of functions with 
compact supports in  $D$. Then the scale of Sobolev spaces can be extended for negative \textcolor{red}{smoothness indexes, too.}  
Namely, $H^{-s} (D)$ can be identified with the dual of $H^s_0 (D)$
with respect to the pairing induced by $(\cdot, \cdot)_{L^2 (D)}$.

If the boundary $\partial D$ of the domain $D$ is sufficiently smooth, 
then, using the standard volume form $d\sigma$ on the hypersurface $\partial D$ induced from 
${\mathbb R}^n$, we may consider the Sobolev and the Sobolev-Slobodeckij spaces  
$H^s(\partial D)$ on $\partial D$. 

In this section we recall both classical and recent results related to 
elliptic  and parabolic differential  operators. With this purpose, recall that a 
linear (matrix) differential operator 
$$
A (x,\partial) = \sum_{|\alpha|\leq m} A_\alpha (x) \partial^\alpha
$$  
 of order $m$ and with $(l\times k)$-matrices $A_\alpha (x) $ having 
entries from $C^\infty (X)$ on an open set $X$, 
is called an operator  with injective symbol on $X\subset {\mathbb R}^n$ 
if $l\geq k$ and for its principal symbol 
$$
\sigma(A) (x,\zeta) = \sum_{|\alpha|= m} A_\alpha (x) \zeta^\alpha
$$ 
we have 
$
\mathrm{rang} \, ( \sigma(A) (x,\zeta) )=k \mbox{ for any } x\in X , \zeta \in {\mathbb R}^n 
\setminus \{0\}.
$ 
An operator $A$  is called (Petrovsky) elliptic, if $l=k$ and 
its symbol is injective.

An operator $L (x,\partial)$ is called strongly elliptic if 
it is elliptic, its order $2m$ is even and there is a positive constant $c_0$ such that 
$$
(-1)^{m}\Re{\, (w^*\sigma(L)} (x,\zeta) \,w)\geq c_0 |\zeta|^{2m} |w|^2
\mbox{ for any } x\in X , \zeta \in {\mathbb R}^n , 
w \in {\mathbb C}^k
$$
where $w^* = \overline w^T$ and $w ^T$ is the transposed vector 
for $w \in {\mathbb C}^k$.

Denote by $\nabla$ the gradient operator  and by 
${\rm div}$ the divergence operator in ${\mathbb R}^n$. 
Obviously, the principal symbol of the operator $\nabla$ is injective. 
Let $M (x)$ be a $(n\times n)$ symmetric non-degenerate  matrix with smooth real entries, 
such that there is a constant $c_0$ providing 
\begin{equation} \label{eq.M.pos}
\zeta \cdot  M (x) \zeta =  \zeta^T M (x)\zeta  \geq c_0 |\zeta|^2 \mbox{ for each } \zeta \in 
{\mathbb R}^n\setminus \{0\} \mbox{ and  each } x\in \overline X.
\end{equation} 
Then the differential operator  
$$
-\Delta_M = -\mbox{div} M \nabla = -\nabla \cdot M \nabla
$$
is elliptic and strongly elliptic on $X$ with 
$\sigma(\Delta_M) (x,\zeta) = -\zeta \cdot  M (x) \zeta$.


Next, we need suitable boundary operators.

\begin{definition} A set of linear differential 
operators $\{B_0,B_1, \dots B_{m-1}\}$ is called a $(k\times k)$ 
Dirichlet system of order $(m-1)$ 
on $\partial D$ if 
1) the operators are defined in a neighbourhood of $\partial D$; 
2) the order of the differential operator $B_j $ equals to $j$; 
3) the map $ \sigma (B_j) (x,\nu (x)) :{\mathbb C}^k \to {\mathbb C}^k$ 
is bijective for each $x \in \partial D$, where 
$\nu (x)$ will denote the outward normal vector to the hypersurface $\partial D$
at the point $x\in \partial D$.
\end{definition}

The simplest Dirichlet pair is the pair $\{1, \frac{\partial}{\partial \nu} \}$, 
where $\frac{\partial}{\partial \nu} $ is the normal derivative with respect to 
$\partial D$.  If we denote by 
$\partial _{\nu ,M}$ the so-called co-normal derivative with respect to $\Delta_M$ and 
set 
$
\partial _{\nu ,M} = \nu^T M \nabla 
$ 
then $\{1, \partial _{\nu ,M} \}$ is 
also a Dirichlet pair, under  assumption \eqref{eq.M.pos}. 

According to the Trace Theorem, see for instance 
\cite[Ch.~1, \S~8]{LiMa72} and \cite{McL00}, if $\partial D\in C^{s}$, $s\geq m\geq 1$ then 
for each $s \in \mathbb N$, $s\geq 2$, each operator 
$B_j$ induces a bounded linear operator
\begin{equation*} 
B_j: H^s (D) \to H^{s-j/2} (\partial D).
\end{equation*}
Thus, the Dirichlet systems are widely used to formulate boundary value problems. 

Now let us discuss the Existence and Uniqueness 
Theorems for four boundary value problems that are essential for 
our approach to the models of the Electrocardiography, considered 
in the next sections. 

We begin with   
the Dirichlet Problem related to strongly elliptic operators. 

\begin{problem} \label{pr.Dir}
 Given pair 
$g \in H^{s-2m} (D)$ and $\oplus_{j=0}^{m-1} u_j \in 
\oplus_{j=0}^{m-1}  H^{s-j-1/2} (\partial D)$ find, if possible, a 
function $u \in H^{s} (D)$ such that 
\begin{equation} \label{eq.Dirichlet.Laplacian}
\left\{ \begin{array}{lll}
L  u =g & {\rm in} & D,\\
\oplus_{j=0}^{m-1} B_j u= \oplus_{j=0}^{m-1} u_j & {\rm on} & \partial D.\\
\end{array}
\right.
\end{equation}
\end{problem}
The problem can be treated in the framework of operator theory in Banach spaces, regarding 
\eqref{eq.Dirichlet.Laplacian} as operator equation 
with the 
linear bounded operator 
$$
(L, \oplus_{j=0}^{m-1} B_j) : H^{s} (D) \to H^{s-2m} (D) \times 
\oplus_{j=0}^{m-1}  H^{s-j-1/2} (\partial D), \, s\geq m. 
$$
Recall that a problem related to operator equation 
$
R u =f
$ 
with a linear bounded operator $R: X_1 \to X_2$ in Banach spaces $X_1, X_2$ has the Fredholm 
property, if the kernel ${\rm ker}(R)$ of the operator $R$ and 
the co-kernel ${\rm coker}(R)$ (i.e. 
the kernel ${\rm ker}(R^*)$ of its adjoint operator $R^*: X_2^* \to X_1^*$)
are finite-dimensional vector spaces and the range of the operator $R$ is closed in $X_2$.  

\begin{theorem} \label{t.Dirichlet.M}
Let $L$ be a strongly elliptic differential operator of order $2m$, $m\geq 1$, 
with smooth coefficients  in a neighbourhood $X$ of $\overline D$,  
$\partial D\in C^s$, $s\geq m$ and 
$B=\{B_0,B_1,\dots, B_{m-1}\}$ be a Dirichlet system of order $(m-1)$ 
on $\partial D$. Then Problem \ref{pr.Dir} has the Fredholm property. Moreover if $L$ 
is formally non-negative and has real analytic coefficients a neighbourhood $X$ of 
$\overline D$,  then Problem \ref{pr.Dir} has one and only one solution.
\end{theorem}

\begin{proof} See, for instance, \cite{Mikh}, \cite[Ch. 5]{Roit96} or elsewhere.
\end{proof}

\begin{corollary} \label{c.Dirichlet.M}
Let $\partial D\in C^s$, $s\geq 1$ and let 
$M (x)$ be a $(n\times n)$ symmetric non-degenerate  matrix with smooth real entries 
satisfying \eqref{eq.M.pos}. Then for each pair 
$g \in H^{s-2} (D)$ and $u_0 \in H^{s-1/2} (\partial D)$ there is unique 
function $u \in H^{s} (D)$ such that 
\begin{equation*} 
\left\{ \begin{array}{lll}
\Delta_M  u =g & {\rm in} & D,\\
 u=  u_0 & {\rm on} & \partial D.\\
\end{array}
\right.
\end{equation*}
\end{corollary}


Now we recall the Existence and Uniqueness 
Theorem for the interior Neumann Problem related to $\Delta _M$.

\begin{problem} \label{pr.Neu} 
Given  pair 
$g \in H^{s-2} (D)$ and $u_1 \in H^{s-3/2} (\partial D)$, find, if possible,  
a function $u \in H^{s} (D)$ such that 
\begin{equation*} 
\left\{ \begin{array}{lll}
\Delta_M u =g & {\rm in} & D,\\
\partial _{\nu ,M} u = u_1  & {\rm on} & \partial D.\\
\end{array}
\right.
\end{equation*}
\end{problem}

\begin{theorem} \label{t.Neumann.M}
Let $s \in \mathbb N$, $s\geq 2$, $\partial D\in C^s$ and let $M (x)$ be a $(n\times n)$
symmetric non-degenerate  matrix with smooth real entries satisfying \eqref{eq.M.pos}. Then
Neumann Problem \ref{pr.Neu} is solvable if and only if 
\begin{equation} \label{eq.Green.M.Neumann}
\int_{\partial D}  u_1 d\sigma + \int_{D} g dx =0.
\end{equation}
The null-space of Problem \ref{pr.Neu}  consists of all the constants. 
Moreover, under \eqref{eq.Green.M.Neumann} there is only one solution $u$ satisfying 
\begin{equation} \label{eq.Neumann.normilize}
\int_{\partial D} u(x) d\sigma(x) =0.
\end{equation}
\end{theorem}

\begin{proof} See, for instance, \cite{Simanca1987}.
\end{proof}

The unique solution to Problem \ref{pr.Neu} satisfying 
\eqref{eq.Neumann.normilize} will be denoted by ${\mathcal N}_i (g,u_1)$.

We continue the section with the discussion of the ill-posed 
Cauchy problem for the operator $\Delta_M$. 

\begin{problem}\label{pr.Cauchy.M}
Fix a part $S$ of $\partial D$ and   a Dirichlet pair $B=\{B_0,B_1\}$ 
on $\partial D$. Given  triple 
$g \in H^{s-2} (D)$, $u_0 \in H^{s-1/2} (\partial D)$ 
and $u_1 \in H^{s-3/2} (\partial D)$, find, if possible,  
a function $u \in H^{s} (D)$ such that 
\begin{equation*} 
\left\{ \begin{array}{lll}
\Delta_M u =g & {\rm in } & D,\\
B_0 u = u_0  & \rm{  on }  & S,\\
B_1 u = u_1  & \rm{ on } & S.\\
\end{array}
\right.
\end{equation*}
\end{problem}

As the Cauchy problem is generally ill-posed, the description
of its solvability conditions is rather complicated. 
 It appears that the regularization methods 
(see, for instance, \cite{TikhArsX}) are most effective for  
studying the problem.  However, there are many different ways to 
realize the regularization, see, for instance, 
\cite{Lv1} \cite{MH74}, \cite{KMF91} for the Cauchy problem related 
to the second order elliptic equations. We follow idea of the book \cite{Tark36}, that
gives a rather full description of solvability conditions for the homogeneous 
elliptic equations, combined with the recent results  
\cite{FeSh14} for elliptic complexes. In order to formulate it we need the 
following Green formula. 

\begin{lemma} \label{eq.dual.Dir}
Let $m\in \mathbb N$, $Q$ be an elliptic operator of order $(m-1)$ in a neighbourhood of 
$\overline D$ and $B=\{B_0,B_1, \dots B_{m-1}\}$ be a Dirichlet system of order $(m-1)$ on $
\partial D$. Then there is a  Dirichlet system $\tilde B =\{\tilde B_0, \tilde B_1,\dots 
\tilde B_{m-1} \}$ on $\partial D$ such that 
for all $v \in H^{m} (D)$, $ u\in  H^m (D)$ we have  
\begin{equation} \label{eq.Green.M.B}
\int_{\partial D} \Big( \sum_{j=0}^{m-1}({\tilde B}_{m-1-j} v)^{*} B_j u 
\Big) d\sigma = \int_{D} \Big( v^{*}  Q u - (Q v)^* u  \Big) dx.
\end{equation}
\end{lemma}

\begin{proof} See, for instance, \cite[Lemma 8.3.3]{Tark97}.
\end{proof}

Ostrogradsky-Gauss formula yields 
that for $Q=\Delta_M$ and the Dirichlet pair 
$B=\{B_0=1, B_1= \partial _{\nu M} \}$ we have the dual 
Dirichlet pair $\tilde B= \{\tilde B_0=1, \tilde B_1= \partial _{\nu M} \}$.

Next, if we assume that the matrix $M$ has real analytic entries and 
satisfies \eqref{eq.M.pos}
we note that all the solutions $u$ to equation $\Delta_M w =0$ in an open set  $U
\subset {\mathbb R}^n$ are real analytic there. Hence
it admits a bilateral (left and right) fundamental solution $\varphi_M (x,y)$, 
see, for instance, \cite[\S 2.3]{Tark95a}. In particular, the following 
Green formula holds true: for each $u \in H^{2} (D)$ we have 
\begin{equation*} 
\chi_D u = {\mathcal G}^{(B)}_{M,\partial D} (B_0u,B_1 u ) + T_{M,D}  (\Delta_M  u),
\end{equation*}
where $\chi_D$ is the characteristic function of the (bounded) domain $D$ in ${\mathbb R}^n$, 
$$
T_{D,M} (g) (x) = \int_D \varphi_M (x,y) g (y) dy,
$$
$$
{\mathcal G}^{(B)}_{M,S} (u_0,u_1 ) = \int_S \Big(
u_0  (y) \tilde B_1 (y)  \varphi_M (x,y)-  u_1 (y) 
\tilde B_0 (y) \varphi_M (x,y)\Big) d\sigma (y)
$$
with a hypersurface $S$ and $x \not \in S$. 

Let us formulate a solvability criterion for Problem \ref{pr.Cauchy.M} under  
reasonable assumptions on $S$. Namely, let us assume that 
$S$ is a relatively open subset of $\partial D$ with a smooth boundary $\partial S$. 
Then for each pair   $u_0 \in H^{s-1/2} (S)$, 
$u_1 \in H^{s-3/2} (S)$ there are functions 
$\tilde u_0 \in H^{s-1/2} (\partial D)$, $u_1 \in H^{s-3/2} (S)$, 
such that $\tilde u_0 = u_0$, $\tilde u_1 = u_1$ on $S$.

Let us fix a domain $D^+$ such that $D\cap D^+=\emptyset$ and the set $G = D\cup S \cup D^+$ is
 a piece-wise smooth domain. We denote by 
$({\mathcal G}^{(B)}_{M,\partial D} (u_0, u_1))^+$ the restriction 
of the potential ${\mathcal G}^{(B)}_{M,\partial D} (u_0, u_1)$ onto $D^+$ and 
similarly for the potential $T_{M,D} (g)$. Obviously, 
$$
\Delta_M ({\mathcal G}^{(B)}_{M,\partial D} (u_0, u_1))^+ = 
\Delta_M  (T_{M,D} (g))^+ =0 \mbox{ in } D^+
$$
as a parameter dependent integral. 

\begin{theorem} \label{t.Cauchy.M}
 Let $s \in \mathbb N$, $s\geq 2$,  $\partial D\in C^s$ 
and the matrix $M$ have real analytic entries and satisfy \eqref{eq.M.pos}. 
If $\partial D \setminus S$ has at least one interior point in the relative 
topology then Problem \ref{pr.Cauchy.M} is densely solvable. 
If $S$ is a relatively open subset of $\partial D$ with a smooth boundary $\partial S$ then
Cauchy Problem \ref{pr.Cauchy.M} has no more than one solution. It 
is solvable if and only if there is 
a function ${\mathcal F}\in H^s (G)$ satisfying 
$\Delta_M {\mathcal F} = 0$ in  $G$
and such that 
$$
{\mathcal F}=({\mathcal G}^{(B)}_{M,\partial D} (\tilde u_0, \tilde u_1))^+ 
+ (T_{M,D} g)^+ \mbox{ in } D^+.
$$
Besides, the solution $u$, if exists, is given by the following formula
\begin{equation} \label{eq.sol.Cauchy}
u  = {\mathcal G}^{(B)}_{M,\partial D} (\tilde u_0, \tilde u_1) +  T_{M,D} g - 
{\mathcal F} \mbox{ in } D.
\end{equation}
\end{theorem}

\begin{proof} See, for instance, \cite[Theorems 2.8 and 5.2]{ShTaLMS} for the case $g=0$ and \cite{FeSh14} for $g\ne 0$.
\end{proof}

At the end of the section we give some information about parabolic theory. 
With this purpose, let $\Omega_T$ be the cylinder domain $\Omega \times (0,T)$ with the base 
$\Omega$ and the time interval $(0,T)$. Let us denote by 
 $H^{2s,s} (\Omega_T)$, $s \in  {\mathbb Z}_+$, 
anisotropic (parabolic) Sobolev spaces, see, for instance, 
\cite{LadSoUr67}, i.e. the set of such measurable functions 
 $u$ on  $\Omega_T$ that the partial derivatives 
$\partial^j_t \partial^{\alpha}_x u$ 
belong to the Lebesgue space $L^{2} (\Omega_T)$ 
for all multi-indexes $(\alpha,j) 
\in {\mathbb Z}_+^{n} \times {\mathbb Z}_+$ satisfying $|\alpha|+2j \leq 2s$. 
This is a Hilbert space with the natural inner product $(u,v)_{H^{2s,s} (\Omega_T)}$. 

We will also use the so-called Bochner spaces of functions depending on 
$(x,t)$ over $\Omega_T$. Namely, if $\mathcal B$ is a Banach space (possibly, a space 
of functions over $\Omega$) and     $p \geq 1$, we denote by  
$L^p ([0,T],{\mathcal B})$ the Banach space of 
measurable maps   $u : [0,T] \to {\mathcal B}$ with the standard norm, 
see, for instance, \cite[Ch. \S 1.2]{Lion69}. 
 
Similarly to elliptic theory, one use often integral representations 
in parabolic theory, too. 
Consider the differential  operator ${\mathcal L}_M = \partial_t + \Delta_M$ and 
a more general differential  operator 
$$
{\mathcal L} = {\mathcal L}_M + \sum_{j=1}^n a_j (x) \partial_j + a_0 (x)
$$
with variable coefficients $a_j (x)$, $0\leq j\leq n$. 
As the operator $\Delta_M$ is strongly elliptic then the operator 
${\mathcal L}$ is parabolic, see, for instance, \cite{Mikh}, \cite{frid}. 
In the particular case $M=I$, we obtain the heat operator ${\mathcal L}_{M}
= \partial_t - \Delta $. 

Under rather mild assumptions on the coefficients $M$ and $a_j$, $0\leq j\leq n$, 
the parabolic differential operator ${\mathcal L}$ admits a fundamental solution 
$\Psi_{\mathcal L}$. In particular, it is the case if the coefficients are constant 
or real analytic and bounded at the infinity, see, 
for instance, \cite[\S 1.5, Theorem 2.8]{eid} or \cite[Ch.1, \S 1--5]{frid}.

\begin{example} \label{ex.fs.heat}
Let $M$ be a non-degenerate matrix with real constant entries.
If we denote by $M^{-1}$ the inverse matrix for $M$ then   the kernel
\begin{equation*}
\Psi_{{\mathcal L}_M}(x,y,t,\tau)=
\begin{cases}
\frac{e^{-\frac{(x-y)^T M^{-1} (x-y)}{4 ( t-\tau)}} }{\left(2\sqrt{\pi {\rm det} (M) \, (t-\tau) }\right)^n } & 
\mbox{ if } t>\tau,\\ 0 & \mbox{ if } t\leqslant   \tau ,
\end{cases} 
\end{equation*} 
is the 
the fundamental solution to the operator ${\mathcal L}_M$, see 
\cite[Ch.1, \S 2]{frid}.
\end{example}

Assume that the the parabolic differential operator ${\mathcal L}$ admits a fundamental solution 
$\Psi_{\mathcal L}$. Denote  by $S$ a relatively open subset of $\partial\Omega$ 
and set $S_T = S \times (0,T)$. For functions 
$g \in L^2 (\Omega_{T})$, $w \in L^2 ([0,T], H^{3/2} (\partial \Omega))$, $v \in L^2 ([0,T], 
H^{1/2} (\partial \Omega))$, $h \in H^{1/2}(\Omega)$ we introduce the following potentials:  
\begin{equation*} 
I_{\Omega} (h) (x,t)= \int\limits_{\Omega}\Psi_{\mathcal L}(x,y, t,0)h(y) dy,  
$$
$$
G_{\Omega} (f) (x,t)=\int\limits_{0}^t\int\limits_\Omega \Psi_{\mathcal L}(x,y,\ t,\tau)
g(y, \tau)dy d\tau, 
\end{equation*}
\begin{equation*} 
V_{S} (v) (x,t)=\int\limits_{0}^t\int\limits_{S} \tilde B_0 (y)
\Psi_{\mathcal L}(x,y, t,\tau) v(y, \tau)
ds(y)d\tau, 
\end{equation*}
\begin{equation*} 
W_{S} (w) (x,t)= - \int\limits_{0}^t\int\limits_{S} \tilde B_1 (y) 
\Psi_{\mathcal L}(x,y, 
t,\tau)w(y, \tau)ds(y) d\tau, 
\end{equation*}
(see, for instance,  
\cite[Ch. 1, \S 3 and Ch. 5, \S 2]{frid}), \cite[Ch. 6, \S 12]{svesh}m where 
$\tilde B = (\tilde B_0, \tilde B_1)$ is the dual Dirichlet pair 
for the elliptic operator 
$$
\Delta_M + \sum_{j=1}^n a_j  \partial_j + a_0 
$$ 
and the Dirichlet pair $B=(1, \partial _{\nu,M})$, see Lemma \ref{eq.dual.Dir}. 
The potential $I_{\Omega} (h)$ is sometimes called \emph{Poisson type integral} and 
the functions $G_{\Omega} (g)$,  $V_{S} (v)$, $W_{S} (w)$ are often referred to 
as heat potentials or, more precisely, \emph{Volume Parabolic Potential}, \emph{Single Layer 
Parabolic Potential} and \emph{Double Layer Parabolic 
Potential}, respectively. By the construction, 
all these potentials are (improper) integral depending on the parameters $(x,t)$.


The theory of boundary value problem for parabolic operators in cylinder 
domains are closely related to the elliptic theory. However, we consider a non-standard 
Cauchy problem for parabolic operators in cylinder domain $\Omega_T$ 
with the Cauchy data on its lateral surface $\partial\Omega \times [0,T]$, 
see, for instance, \cite{PuSh}, \cite{KuSh}. 

\begin{problem} \label{pr.Cauchy.heat} 
Given  $u_1 \in L^2 ([0,T], H^{3/2} (S))$, 
$u_2 \in L^2 ([0,T], H^{1/2} (S))$, $g \in L^2 (\Omega_T)$, find 
$u \in H^{2,1} (\Omega_T)$, satisfying
\begin{equation*} 
\left\{ \begin{array}{lll}
{\mathcal L} u =g & {\rm in} & \Omega_T,\\
 u = u_1  & {\rm on} & S_T,\\
\partial_{\nu,M} u = u_2  & {\rm on} & S_T.\\
\end{array}
\right.
\end{equation*}   
\end{problem}
This problem is generally ill-posed. It may be treated similarly to 
the ill-posed Cauchy problem for elliptic equations, \cite{KuSh}, \cite{PuSh}. 
Let us fix a domain $\Omega^+$ such that $\Omega\cap \Omega^+=\emptyset$ and the set 
$D = \Omega\cup S \cup \Omega^+$ is
 a piece-wise smooth domain. We denote by 
$(I_{\Omega} (u))^+$ the restriction 
of the potential $I_{\Omega} (u)$ onto $\Omega^+$ and 
similarly for the potentials $W_{S} (w)$, $V_{S} (v)$ and $G_{\Omega} (f)$. Obviously, 
$$
{\mathcal L} (I_{\Omega} (u))^+ = {\mathcal L} (G_{\Omega} (f))^+ =
{\mathcal L} (V_{S} (v))^+ = {\mathcal L} (W_{S} (w))^+ 
 =0 \mbox{ in } (\Omega^+)_T
$$
as a parameter dependent integrals. 

\begin{theorem} \label{t.Cauchy.heat}
 Let $s \in \mathbb N$, $s\geq 2$, and $\partial \Omega\in C^s$. If $S$ is a relatively 
open subset of $\partial \Omega$ with a smooth boundary $\partial S$ then
Cauchy Problem \ref{pr.Cauchy.heat} has no more than one solution. It 
is solvable if and only if there is a function ${\mathcal F}\in H^{2,1} (D_T)$ satisfying 
${\mathcal L} {\mathcal F} = 0$ in  $D_T$ and such that 
$$
{\mathcal F}=(G_{\Omega} (f))^+ 
 + (W_{S} (u_1))^+  +(V_{S} (u_2))^+ \mbox{ in } (\Omega^+)_T.
$$
Besides, the solution $u$, if exists, is given by the following formula
\begin{equation*} 
u  = G_{\Omega} (f)
 + W_{S} (u_1)  +V_{S} (u_2) 
- {\mathcal F}  \mbox{ in } D_T.
\end{equation*}
\end{theorem}

\begin{proof} See, for instance, \cite{KuSh}, \cite{PuSh}
for similar problems related to various parabolic 
differential operators  in the anisotropic H\"older spaces and Sobolev spaces.
\end{proof}

\section{A steady bidomain model of the heart}
\label{s.bidomain}

Following a standard scheme (see, for instance, \cite{geselowitz1983}, 
\cite[\S 2.2.2]{2}, or elsewhere) 
let us consider the myocardial domain $\Omega_m$ being surrounded by a volume conductor 
$\Omega_b $. The total domain, including the myocardium and the torso $\Omega = \Omega_b \cup 
\overline \Omega_m$, where $\overline \Omega_m$ is the closure of heart 
domain,  is surrounded by a non-conductive medium (air). We assume that 
 the conductivity matrices $M_i$, $M_e$, $M_b$ 
of the intra-, extracellular and extracardiac media 
satisfy \eqref{eq.M.pos} and that their entries are real analytic 
 in some neighbourhoods of $\overline \Omega_m$ and $\overline \Omega_b$, respectively; 
of course, if we assume that these media are homogeneous and isotropic, then the entries 
will be just constants.

Then the differential operators
$$
\Delta_e = -\nabla \cdot M_e \nabla , \, \Delta_i = -\nabla \cdot M_i \nabla, \, 
\Delta_b = -\nabla \cdot M_b \nabla 
$$
are elliptic and strongly elliptic and admit  the bilateral fundamental solutions, say, 
$\varphi_i $, $\varphi_e $, $\varphi_b  $,  over some neighbourhoods of 
$\overline \Omega_m$ and $\overline \Omega_b$, respectively.

If $u_i, u_e \in H^2 (\Omega_m)$, $u_b \in H^2 (\Omega_b) $ are intra-, extracellular and 
extracardiac potentials, respectively, then the  intra-, extracellular and  extracardiac 
currents are given by
$$
J_i = -M_i \nabla u_i, \, J_e = -M_e \nabla u_e, \, u_b = -M_b \nabla u_b, 
$$
respectively. As the intracellular charge $q_i$ and the extracellular $q_e$ should be 
balanced in the heart tissue, we arrive at the following equations
\begin{equation} \label{eq.balance.charge}
\frac{\partial (q_i+q_e)}{ \partial t} =0,
\end{equation}
\begin{equation} \label{eq.balance.current.i}
-\Delta_i u _i = \frac{\partial q_i}{ \partial t} +\chi I_{\rm ion},
\end{equation}
\begin{equation} \label{eq.balance.current.e}
-\Delta_e u _e = \frac{\partial q_e}{ \partial t} -\chi I_{\rm ion},
\end{equation}
where $I_{\rm ion}$ is the ionic current across the membrane and 
$\chi I_{\rm ion}$ is  ionic current per unit tissue. 
Of course, the potentials $u_e$ and $u_i$ are actually defined 
on different domains: extracellular and intracellular spaces in $\Omega_m$, respectively.
Thus, the last equations reflect the fact that a homogenization procedure 
is at the bottom of the bidomain model. 

Then, combining  \eqref{eq.balance.charge}, \eqref{eq.balance.current.e}
\eqref{eq.balance.current.i} we obtain the conservation law for the total 
current $(J_i+ J_e)$:
$$
\Delta_i u_i + \Delta_e  u_e  = 0 \mbox{ in } \Omega_m.
$$

Next,  denote  by $\nu_i$, $\nu_e$  the outward normal 
vectors to the surfaces of the heart and body volume ($\Omega_m$ and $\Omega_b$), 
respectively. In the heart surrounded by a conductor, the normal component of the total 
current  should be continuous across the boundary of the heart: 
\begin{equation}\label{eq.balance.boundary}
\nu_i \cdot (J_i +J_e) = \nu_i \cdot J_b.
\end{equation}

Taking in account the current behaviour at the torso,  
 we arrive at a steady-state version of the bidomain model \cite{geselowitz1983}:
\begin{align}
\Delta_i u_i + \Delta_e  u_e  = 0 \mbox{ in } \Omega_m, \label{eq7} \\
\Delta_b  u_b  = 0\mbox{ in } \Omega_b ,\label{eq8} \\
	u_e  = u_b   \mbox{ on } \partial \Omega_m , \label{eq9} \\
	\nu_i \cdot (M_e \nabla u_e ) 
	= - \nu_e \cdot (M_b \nabla u_b)\mbox{ on }  \partial \Omega_m  \label{eq10} \\
	\nu_i \cdot (M_i \nabla u_i)  =0 \mbox{ on } \partial \Omega_m,  \label{eq11} \\
	\nu_e \cdot (M_b \nabla u_b)  = 0 \mbox{ on } \partial \Omega,  \label{eq12}
\end{align}
where \eqref{eq10}, \eqref{eq11} are consequences of \eqref{eq.balance.boundary} and 
the assumption that the intracellular domain is completely insulated. 

However, this steady model is often supplemented by an 
evolutionary part. For example,  \eqref{eq.balance.current.i}, 
\eqref{eq.balance.current.e} imply
\begin{equation*} 
-\Delta_i u _i + \Delta_e u _e = \frac{\partial (q_i-q_e)}{ \partial t} +2\chi I_{\rm ion}.
\end{equation*}
On the other hand, the transmembrane potenitial $v=u_i-u_e$ satisfies 
\begin{equation*} 
u_i-u_e = \frac{1}{2} \frac{q_i-q_e}{\chi \, C_m}
\end{equation*}
where $C_m$ is the capacitance of the cell membrane. 
Thus, using the last two identities 
we arrive 
at the so-called cable equation
\begin{equation} \label{eq.balance.cable.0}
\frac{1}{2\chi}\Big(-\Delta_i u _i + \Delta_e u _e\Big) = 
 C_m \frac{\partial (u_i-u_e)}{ \partial t} + I_{\rm ion} 
\mbox{ in } \Omega_m \times (0,T),
\end{equation}
see, for instance, \cite[\S 2.2.2]{2}. There are 
other advanced and complicated relations that can be added to the model. 

In this section we will discuss the steady part of the bidomain model only, 
considering the following problem:

\begin{problem} \label{pr.inverse.inside}
Let the values of electrical 
potential $u_b$ on the boundary of the body domain be known:
\begin{equation} 
	u_b = f \mbox{ on }  \partial \Omega,
	\label{eq6}
\end{equation}
where $f\in H^{3/2} (\partial \Omega)$. Under these conditions we seek for the intracellular 
potential $ u_i \in H^{2} (\Omega_m)$ and extracellular potential 
$u_e \in H^{2} (\Omega_m)$ and extracardiac potential $u_b\in H^{2} (\Omega_b)$ 
satisfying equations \eqref{eq7}, \eqref{eq8} and 
 boundary conditions \eqref{eq9}-\eqref{eq12}.
\end{problem}

The non-uniqueness of solutions to Problem \ref{pr.inverse.inside} 
was established in \cite{2K} in specially constructed 
Hardy type spaces under the following 
restrictive assumptions:

1) all the media are homogeneous and isotropic;

2) the matrices $M_i$, $M_e$,  
are proportional, i.e. 
\begin{equation} 
\label{eq.prop}
M_e = \lambda M_i \mbox{ with some positive number } \lambda.
\end{equation}
 In particular, this means that 
a linear change of variables reduces the consideration to the situation 
where 
\begin{equation} \label{eq.scalar}
\Delta_i = -\sigma_i \Delta, \, \Delta_e = -\sigma_e \Delta, \, \lambda= \frac{\sigma_e}{\sigma_i}
\end{equation}
and $\sigma_i $, $\sigma_e $, 
are positive numbers characterizing the 
 electrical conductivity of the corresponding media. 

Let us describe the null-space of Problem \ref{pr.inverse.inside} in 
a more general situation. With this purpose,  we use the following calibration 
assumption that always is achievable for isotropic conductivity:
there is a constant $c_0$ such that 
\begin{equation} \label{eq.calibration}
\int_{\partial \Omega_m} (u_i + c_0 u_e)(y) d\sigma (y) =0.
\end{equation}

\begin{proposition} \label{p.bidomain.null}
The null-space of Problem \ref{pr.inverse.inside} consists of all the triples $u_i, u_e
\in H^2 (\Omega_m)$, $u_b \in H^2 (\Omega_b)$  satisfying the following conditions: 
\begin{equation} \label{eq.nullspace.inside}
	\left\{ 
	\begin{array}{ccccc}
	u_b  & = & 0 & {\rm in } & \Omega_b,\\ 
	u_e  & = & u & {\rm in } & \Omega_m,\\
	 u_i  & = -& {\mathcal N}_i (\Delta_e u,0) +c & {\rm in } & \Omega_m,\\
\end{array}
\right.
\end{equation}
where ${\mathcal N}_i $ is the Neumann operator related to $\Delta_i$, 
$c$ is an arbitrary constant and $u $ is an arbitrary function 
from $ H^2_0 (\Omega_m)$. If 
calibration assumption \eqref{eq.calibration} holds for a pair $u_i$, $u_e$ from 
the null-space then the constant $c$ in \eqref{eq.nullspace.inside} equals to zero.  
\end{proposition}

\begin{proof} Indeed, let the triple $(u_b, u_i,u_e) \in 
H^2 (\Omega_b) \times H^2 (\Omega_m) \times H^2 (\Omega_m)  $ belong to 
the null-space of Problem \ref{pr.inverse.inside}. Hence 
$f\equiv 0$ on $\partial \Omega$ and then 
$u_b \equiv 0$ in $\Omega_b$ because of  Theorem \ref{t.Cauchy.M} 
for the Cauchy problem. Of course, using \eqref{eq9}, \eqref{eq10}, we obtain 
$$
u _e =u_b = \nu_i \cdot (M_e \nabla u_e) =- \nu_e \cdot (M_b \nabla u_b) =0 
\mbox{ on } \partial \Omega _m 
$$ 
for $u_e \in H^2 (\Omega_m)$. Then, according to \cite{HedbWolf1}, 
$u_e \in H^2_0 (\Omega_m)$. The function $u_i$ satisfies \eqref{eq7} and \eqref{eq11} and 
then, according to Theorem \ref{t.Neumann.M}, 
this means precisely 
$$
u_i =- {\mathcal N}_i (\Delta_e u_e,0) +c
$$
 with 
an arbitrary constant $ c$. 

Thus, any triple  $(u_b, u_i,u_e) \in 
H^2 (\Omega_b) \times H^2 (\Omega_m) \times H^2 (\Omega_m)  $, belonging to 
the null-space of Problem \ref{pr.inverse.inside}, has the form 
as in \eqref{eq.nullspace.inside} with a  constant $c$ and a
function $ u=u_e\in H^2_0 (\Omega_m)$. 

Let a triple  $(u_b, u_i,u_e) \in 
H^2 (\Omega_b) \times H^2 (\Omega_m) \times H^2 (\Omega_m)  $ have the form 
as in \eqref{eq.nullspace.inside} with an arbitrary  constant $c$ and an arbitrary 
function $ u\in H^2_0 (\Omega_m)$. 

Then, obviously, $f\equiv 0$ on $\partial \Omega$. 
Moreover, using Green formula \eqref{eq.Green.M.B}, 
we easily obtain
$$
-\int_{\Omega_m}\Delta_e u dy = \int_{\partial \Omega_m} 1 \, \nu_e \cdot  M_e \nabla u 
d \sigma =0
$$
because $u \in H^2 _0 (\Omega_m)$. Hence Theorem \ref{t.Neumann.M} implies that 
 there is a potential $w$ satisfying Neumann Problem \ref{pr.Neu} for 
the operator $\Delta_i$:
\begin{equation*} 
\left\{ \begin{array}{lll}
\Delta_i w =- \Delta_e u & {\rm in} & \Omega_m,\\
\nu_i  \cdot M_i \nabla w= 0  & {\rm on} & \partial\Omega_m.\\
\end{array}
\right.
\end{equation*}
According to Theorem \ref{t.Neumann.M}, 
the general form of such a solution is precisely 
\begin{equation} \label{eq.sol.ui}
u_i =- {\mathcal N}_i (\Delta_e u_e,0) +c
\end{equation}
with a constant $c$. 

If we take $u_i = w$ then 
\begin{equation*} 
	\left\{ 
	\begin{array}{ccccc}
	 u_b  & = & 0 & {\rm in} & \Omega_b,\\
	 u_e  & = & u & {\rm in} & \Omega_m,\\
	\nu_i \cdot (M_e \nabla u_e)  & = & 0 & {\rm on} & \partial \Omega_m,\\ 
	u_b  & = & 0 & {\rm on} & \partial \Omega_m,\\
	\Delta_i u_i  & = & -\Delta_e u & {\rm in} & \Omega_m,\\
	\nu_i \cdot (M_i \nabla u_i)  & = & 0 & {\rm on} & \partial \Omega_m.\\ 
\end{array}
\right.
\end{equation*}

Thus, any triple  $(u_b, u_i,u_e) \in 
H^2 (\Omega_b) \times H^2 (\Omega_m) \times H^2 (\Omega_m)  $ having the form 
as in \eqref{eq.nullspace.inside} with a arbitrary  constant $c$ and an arbitrary 
function $ u\in H^2_0 (\Omega_m)$ belongs to the 
null-space of Problem \ref{pr.inverse.inside}. 

Finally, if calibration assumption \eqref{eq.calibration} is fulfilled then, 
as $u_e =u \in H^2 _0(\Omega_m)$, condition \eqref{eq.Neumann.normilize} yields 
$$
0=\int_{\partial \Omega_m} (u_i + c_0 u_e)(y) d\sigma (y) 
=\int_{\partial \Omega_m} (-{\mathcal N}_i (\Delta_e u_e,0) +c  ) d\sigma (y) = 
 c \int_{\partial \Omega_m} d\sigma (y).
$$
Since the area of the surface $\partial \Omega_m$ is not zero, we conclude that $c=0$. 
\end{proof}

We note that 
a closely related result is presented in \cite[Appendix A]{Nielsen2007}, but it is  
formulated in terms of the transmembrane potential 
$v=u_i-u_e$ instead of the intracellular voltage. 

We are ready to formulate an existence theorem for the bidomain model above.

\begin{theorem} \label{t.bidomain.exists}
Given $f \in H^{3/2} (\partial \Omega_b)$, admitting 
the solution $u_b\in H^2 (\Omega_b)$  to \eqref{eq8}, \eqref{eq12} and \eqref{eq6}, 
there are functions $u_e,u_i\in H^2 (\Omega_m)$
satisfying \eqref{eq7},\eqref{eq9}, \eqref{eq10}, \eqref{eq11}.
Moreover, if calibration assumption \eqref{eq.calibration} holds for a pair $u_i$, $u_e$ 
 then the constant $c$ in \eqref{eq.sol.ui} is uniquely defined by 
\begin{equation}  \label{eq.const}
c= -c_0\Big(\int_{\partial \Omega_m}  d\sigma (y) \Big)^{-1}
\int_{\partial \Omega_m}  u_b (y) d\sigma (y) . 
\end{equation}
\end{theorem}

\begin{proof} We begin with the well known lemma. 

\begin{lemma} \label{l.Dir.right}
Let $\partial D\in C^s$, $s\geq m$ and 
$B=\{B_0,B_1,\dots, B_{m-1}\}$ be a Dirichlet system of order $(m-1)$ 
on $\partial D$. Then for each set $\oplus_{j=0}^{m-1} 
u_j \in \oplus_{j=0}^{m-1} H^{s-j-1/2} (\partial D)$
there is a function $u \in H^s (D)$ such that 
\begin{equation*} 
\oplus_{j=0}^{m-1} B_j u =\oplus_{j=0}^{m-1}u_j \mbox{ on }\partial D.
\end{equation*}
\end{lemma}

\begin{proof} See, for instance, \cite[Lemma 5.1.1]{Roit96}. 
\end{proof}

As $u_b\in H^2 (\Omega_b)$  we see that 
$u_b \in H^{3/2} (\partial \Omega_b) $, $\nu_e \cdot (M_b \nabla u_b) \in 
 H^{1/2} (\partial \Omega_b) $.  Applying Lemma \ref{l.Dir.right} for 
the operator $\Delta_M$ and the Dirichlet pair 
$B=\{B_0=1, B_1= \partial _{\nu M} \}$, we may find 
a function $u_e \in H^2 (\Omega_m)$ satisfying \eqref{eq9}, \eqref{eq10}.
For example, one may take $u_e$ as 
the unique solution $u\in H^2 (\Omega_m)$ to Dirichlet problem
\begin{equation} \label{eq.Dirichlet.Laplacian.Q}
\left\{ \begin{array}{lll}
Q u =g & {\rm in} & \Omega_m,\\
 u = u_b & {\rm on} & \partial \Omega_m,\\
\nu_i \cdot (M_e \nabla u ) 
	= - \nu_e \cdot (M_b \nabla u_b)  &  \rm{on} &  \partial \Omega_m   \\
\end{array}
\right.
\end{equation}
with an arbitrary function $g\in H^{-2} (\Omega_m)$ and an arbitrary strongly formally 
non-negative elliptic operator $Q$ of the fourth order and with real analytic 
coefficients in a neighbourhood of $\overline \Omega_m$ (see Theorem \ref{t.Dirichlet.M}). 
In particular, under this assumptions Problem \eqref{eq.Dirichlet.Laplacian.Q}
admits a unique Green function, say ${\mathcal G}(x,y)$ and thus $u_e$ may be expressed via an 
integral formula  
\begin{equation} \label{eq.sol.ue}
u_e (x) = \int_{\partial \Omega_m} \Big( B_3 (y)  {\mathcal G} (x,y) u_b (y) -
 B_2 (y) {\mathcal G} (x,y)  \nu_e \cdot (M_b \nabla u_b) (y) \Big) d\sigma + 
\end{equation}
\begin{equation*}
\int_{\Omega_m} {\mathcal G}(x,y) g(y) \, dy , \, x\in \Omega_m, 
\end{equation*}
 where  $(1, \nu_i \cdot (M_e \nabla ), B_2, B_3 )$ 
is a Dirichlet quadruple of the third order satisfying 
\begin{equation*}
 \int_{\partial \Omega_m} \Big( B_3v (y)  u (y) +
 B_2 v(y)  \nu_i \cdot (M_e \nabla u)  \Big) d\sigma =  
\end{equation*}
$$
\int_{\Omega_m} \Big( v (Q u) - (Q^*v) u \Big) dy
$$
for all $u \in H^4 (\Omega_m)$, $v \in H^4 (\Omega_m) \cap 
H^2_0 (\Omega_m)$. 
For example, one may take $Q=\Delta_e^2 $, 
$B_2 = \Delta_e$, $B_3 =  -\nu_i \cdot (M_e \nabla \Delta_e )$
 because
 $\Delta_e=\Delta^*_e$ and hence the operator 
$$
\Delta_e^2 = (\Delta_e)^* \Delta_e
$$ 
is strongly elliptic, formally non-negative and of fourth order. 

 Next, integrating by parts with the use of 
\eqref{eq.Green.M.B}, \eqref{eq8}, \eqref{eq10}, \eqref{eq12},  we obtain
$$
-\int_{\Omega_m} \Delta_e u_e (y) dy = 
\int_{\partial \Omega_m}  \nu_i \cdot (M_e \nabla u _e)d\sigma
= -\int_{\partial \Omega_m}  \nu_i \cdot (M_b \nabla u _b) d\sigma+
$$
$$
\int_{\partial \Omega_b}  \nu _e \cdot (M_b \nabla u _b) d\sigma
= - \int_{\Omega_b} \Delta_b u_b (y) dy =0.
$$
Now Theorem \ref{t.Neumann.M} yields the existence of a function 
$u_i \in H^2 (\Omega_m)$ satisfying 
\begin{equation*} 
\left\{ \begin{array}{lll}
\Delta_i u_i = -\Delta_e u_e & \rm{ in } & \Omega_m,\\
\nu _i \cdot (M_i \nabla u _i)   = 0  & \rm{ on } & \partial \Omega_m.\\
\end{array}
\right.
\end{equation*}
More precisely, Theorem \ref{t.Neumann.M} states that $u_i$ is given by \eqref{eq.sol.ui}
with   an arbitrary constant $ c$. 

Again, if calibration assumption \eqref{eq.calibration} holds for a pair $u_i$, $u_e$ 
 then 
$$
0=\int_{\partial \Omega_m} \Big( c-  {\mathcal N}_i (\Delta_e u_e,0)  +c_0 u_e\Big) d\sigma (y) =
$$
$$
c \, \int_{\partial \Omega_m}  d\sigma (y)  + 
c_0 \int_{\partial \Omega_m} u_b (y) d\sigma (y)  
$$
because of normalising condition \eqref{eq.Neumann.normilize}. 
Thus, the constant $c$ in \eqref{eq.sol.ui} may be uniquely defined by \eqref{eq.const}.
\end{proof}

\begin{example} \label{ex.proportional} 
Of course, in some particular situations we can say much more. For instance,  
\eqref{eq.prop} and \eqref{eq.scalar} are fulfilled, then we have
$$
{\mathcal N}_i (\Delta_e u,0) = \lambda {\mathcal N}_i (\Delta_i u,0) =  \lambda u
$$
for each $u \in H^2_0 (\Omega_m)$. In particular, in this case, 
according to Theorem \ref{t.Neumann.M}, 
\begin{equation} \label{eq.nullspace.inside.prop}
	\left\{ 
	\begin{array}{ccccc}
	 u_e  & = & u & {\rm in} & \Omega_m,\\
	 u_i  & = & -\lambda u + c& {\rm in} & \Omega_m,\\
\end{array}
\right.
\end{equation}
for each pair $u_i$, $u_e$ from 
the null-space of Problem \ref{pr.inverse.inside} where 
$c$ is an arbitrary constant and $u $ is an arbitrary function 
from $ H^2_0 (\Omega_m)$.  Again, if 
calibration assumption \eqref{eq.calibration} holds for a pair $u_i$, $u_e$ from 
the null-space then the constant $c$ in \eqref{eq.nullspace.inside.prop} equals to zero.

As for the Existence Theorem, in this case 
$$
{\mathcal N}_i (\Delta_e u_e,0) = \lambda {\mathcal N}_i (\Delta_i u_e,0) 
=  \lambda u_e -\lambda {\mathcal N}_i (0, \nu_i \cdot (M_b \nabla u _b)).
$$
Thus, formula \eqref{eq.sol.ui} implies 
\begin{equation} \label{eq.sol.ui.prop}
u_i =-\lambda u_e +\lambda {\mathcal N}_i (0, \nu_i \cdot (M_b \nabla u _b)) +c
\end{equation}
where 
$c$ is an arbitrary constant.  Again, if 
calibration assumption \eqref{eq.calibration} holds for the pair $u_i$, $u_e$  
then the constant $c$ may be uniquely defined by \eqref{eq.const}.
\end{example}

Thus, Theorem \ref{t.bidomain.exists} gives a clear path for finding 
the potentials $u_b, u_i, u_e$ and $v$ on the myocardial surface $\partial \Omega_m$:
\begin{enumerate}
\item[(1)] 
given suitable $f \in H^{3/2} (\partial \Omega)$ described in Theorem
\ref{t.Cauchy.M}, find the potential $u_b$ over $\overline \Omega_b$
using formula \eqref{eq.sol.Cauchy} or related formula evoking 
 bases with the double orthogonality property, see
\cite{ShTaLMS} or iteration methods, see \cite{KMF91});
\item[(2)]
choosing suitable  fourth order strongly elliptic operator $Q$  
and function $h\in H^{-2} (\Omega_m)$, find the potential $u_e$ 
with the use of formula \eqref{eq.sol.ue};
\item[(3)]
find the potential $u_i$ 
with the use of formula \eqref{eq.sol.ui};
\item[(4)]
calculate the potential $v=u_i-u_e$ on $\partial \Omega_m$.
\end{enumerate}

From mathematical point of view, 
Proposition \ref{p.bidomain.null}  means 
that the steady part of the bidomain model  
has too many degrees of freedom. More precisely,  
 at least one  equation related to these  
potentials  in $\Omega_m$ is still missing.
 
Thus, staying in the framework of steady models related 
to the elliptic theory,  the proof of Theorem \ref{t.bidomain.exists} suggests to look 
for an additional fourth order strongly elliptic equation
\begin{equation} \label{eq.add}
Q u_e =g \mbox{ in } \Omega_m 
\end{equation}
with a given function $g$ in $\Omega_m$ depending on a patient 
in order to provide the existence and the uniqueness theorem 
for Problem \ref{pr.inverse.inside}.

\begin{corollary} \label{c.bidomain.exists}
Let \eqref{eq.prop} hold true and function $f \in H^{3/2} (\partial \Omega_b)$ admits
the solution $u_b\in H^2 (\Omega_b)$  to \eqref{eq8}, \eqref{eq12} and \eqref{eq6}. 
If $Q$ is a fourth order strongly elliptic operator with smooth 
coefficients over $\overline \Omega_m$  then,
given vector $g \in H^{-2} (\Omega_m)$, problem   
\eqref{eq7}, \eqref{eq9}, \eqref{eq10}, \eqref{eq11}, \eqref{eq.add} has the  
Fredholm property. Moreover, if 
$Q$ is a fourth order formally non-negative strongly elliptic operator 
with real analytic coefficients over $\overline \Omega_m$, then,
given vector $g \in H^{-2} (\Omega_m)$, problem   
\eqref{eq7}, \eqref{eq9}, \eqref{eq10}, \eqref{eq11}, 
\eqref{eq.calibration}, \eqref{eq.add} has 
one and only one solution $(u_i,u_e) \in H^{2} (\Omega_m) \times 
H^{2} (\Omega_m)$.
\end{corollary}

\begin{proof}  Under the hypothesis of this corollary both Dirichlet problem 
 \eqref{eq9}, \eqref{eq10}, \eqref{eq.add}, see, for instance, \cite{Roit96}  and 
Neumann problem \eqref{eq7}, \eqref{eq11}, see, 
for instance, \cite{Simanca1987}, have Fredholm property 
in the relevant Sobolev spaces. Hence the first part of the 
statement of the corollary is proved. 

If we additionally assume that  
$Q$ is a fourth order formally non-negative strongly elliptic operator 
with real analytic coefficients over $\overline \Omega_m$, then,
given vector $g \in H^{-2} (\Omega_m)$,  
Dirichlet problem  \eqref{eq9}, \eqref{eq10}, \eqref{eq.add} has one and only 
one solution $u_e  \in H^{2} (\Omega_m)$.  Moreover, as we have seen in the proof 
of Theorem \ref{t.bidomain.exists}, under calibration condition 
\eqref{eq.calibration}, 
Neumann problem \eqref{eq7}, \eqref{eq11} 
is uniquely solvable, too. Thus, problem 
\eqref{eq7}, \eqref{eq9}, \eqref{eq10}, \eqref{eq11}, 
\eqref{eq.calibration}, \eqref{eq.add} has 
one and only one solution $(u_i,u_e)\in H^{2} (\Omega_m) \times 
H^{2} (\Omega_m)$.
\end{proof}

Of course, the suggestion to add equation \eqref{eq.add} to the 
steady  bidomain model is purely mathematical. 
However, as Maxwell's Electrodynamics Theory shows, often 
a purely mathematical proposal leads to a full solution in Natural Sciences.
Thus  we are just informing the scientific community about the corresponding possibility. 
Let us give an instructive example
illustrating that there is not so much hope that this can  improve essentially the bidomain 
model in a general situation.  Though,  one may hope to construct such an equation 
using specific information on the cardiac tissues or even in a patient specific manner. 

\begin{example} 
Consider Problem \ref{pr.inverse.inside} 
in the situation where assumptions \eqref{eq.prop} 
and  \eqref{eq.scalar} are fulfilled. Next 
we assume that the function $f$ in \eqref{eq6} does not depend on $t$, 
calibration condition \eqref{eq.calibration} is fulfilled and  
that the following electrodynamic relation holds true for steady currents:
\begin{equation*} 
\Delta u  = - \frac{q}{\varepsilon \varepsilon_0}
\end{equation*}
where $q$ is the density of electric charges, 
$u$ is the potential the electric field and $\varepsilon \varepsilon_0>0$ is the  
dielectric constant of the medium. 
In particular, for the potentials $u_i$, $u_e$ we obtain
\begin{equation} \label{eq.current.2}
\Delta u_i  = - \frac{q_i}{\varepsilon \varepsilon_0}, \,\, 
\Delta u_e  = - \frac{q_e}{\varepsilon \varepsilon_0}.
\end{equation}

Hence, substituting \eqref{eq.current.2} into \eqref{eq.balance.current.i} and 
\eqref{eq.balance.current.e} we obtain formulas that can be useful if we need to transform 
evolutionary equations to stationary ones:
\begin{equation*} 
\frac{\partial \Delta u_i} {\partial t}  =
- \frac{1}{\varepsilon \varepsilon_0} \Big( \sigma _i \Delta u_i -\chi I_{\rm ion} \Big), 
\end{equation*}
\begin{equation*} 
\frac{\partial \Delta u_e }{\partial t}  =
- \frac{1}{\varepsilon \varepsilon_0} \Big( \sigma _e \Delta u_e +\chi I_{\rm ion} \Big).
\end{equation*}
Now, taking in account \eqref{eq7}, cable equation \eqref{eq.balance.cable.0} and 
\eqref{eq.sol.ui.prop} (where, obviously, $\lambda=\frac{\sigma_e}{\sigma_i}$) 
we obtain the following chain of  equations in the sense of distributions in 
$ \Omega_m \times (0,T)$:
\begin{equation*} 
-\frac{\sigma_e}{\chi} \Delta^2 u _e = 
 C_m \frac{\partial \Delta (u_i-u_e)}{ \partial t} + \Delta I_{\rm ion} =
\end{equation*}
\begin{equation*}
- \frac{C_m(\sigma_e+\sigma_i)}{\sigma_i } \frac{\partial \Delta u_e}{ \partial t} +
 \Delta I_{\rm ion} =
 \frac{C_m(\sigma_e+\sigma_i)}{\sigma_i \varepsilon \varepsilon_0} 
\Big(\sigma_e\Delta u_e + I_{\rm ion} \Big) +  \Delta I_{\rm ion} =
\end{equation*}
\begin{equation*}
 \frac{C_m(\sigma_e+\sigma_i)}{\sigma_i \varepsilon \varepsilon_0} 
\Big(\sigma_e\Delta u_e + I_{\rm ion} \Big) + 
 \Delta I_{\rm ion} 
\end{equation*}
and, similarly, 
\begin{equation*} 
\frac{\sigma_i}{\chi} \Delta^2 u _i = - \frac{C_m(\sigma_e+\sigma_i)}{\sigma_e 
\varepsilon \varepsilon_0} \Big(\sigma_i\Delta u_i - I_{\rm ion} \Big)+
 \Delta I_{\rm ion} .
\end{equation*}

Therefore
\begin{equation} \label{eq.balance.cable.00.e}
\frac{\sigma_i \sigma_e \varepsilon \varepsilon_0}{\sigma_e+\sigma_i} 
\Delta^2 u _e = 
- \chi C_m \Big(  \sigma_e\Delta u_e + I_{\rm ion}\Big) -
\frac{\chi \sigma_i \sigma_e \varepsilon \varepsilon_0}{\sigma_e+\sigma_i} 
\Delta I_{\rm ion} ,
\end{equation}
\begin{equation*} 
\frac{\sigma_i \sigma_e \varepsilon \varepsilon_0}{\sigma_e+\sigma_i} 
\Delta^2 u _i = 
- \chi C_m \Big(  \sigma_i\Delta u_i + I_{\rm ion}\Big) + 
\frac{\chi \sigma_i \sigma_e \varepsilon \varepsilon_0}{\sigma_e+\sigma_i} 
\Delta I_{\rm ion}.
\end{equation*}
If we are to stay within the framework of linear theory we may 
assume that the ionic current is given by 
\begin{equation} \label{eq.Iion.lin.steady}
I_{\rm ion} (v) = \sum_{j=1}^n a_j \partial_j v + a_0 v + b  
\end{equation} 
with some function  $b \in L^2 (\Omega)$, and some  constants $a_j$, $0\leq j \leq n$.  
Then, as the operator $\Delta^2$ is strongly elliptic, using \eqref{eq.sol.ui.prop} (where, 
obviously, $\lambda=\frac{\sigma_e}{\sigma_i}$) and 
\eqref{eq.balance.cable.00.e} we arrive 
at the fourth order strongly elliptic equation
\begin{equation} \label{eq.add.prop}
\frac{\sigma_i \sigma_e \varepsilon \varepsilon_0}{\sigma_e+\sigma_i} 
\Delta^2 u _e + \chi C_m 	\sigma_e \Delta u_e +
	\frac{C_m(\sigma_e+\sigma_i)}{\sigma_i } -
\end{equation}
$$
 \chi \sigma_e \varepsilon \varepsilon_0 \Delta I_{\rm ion}(u_e) 
= - \frac{\sigma_e}{\sigma_i}\chi C_m  I_{\rm ion} 
({\mathcal N}_i (0, \nu_i \cdot (M_b \nabla u _b))  .
$$
In general, there is little hope that Dirichlet problem  
\eqref{eq.add.prop}, \eqref{eq9}, \eqref{eq10} is uniquely solvable 
because the coefficient $\varepsilon \varepsilon_0 $ is practically 
very small. Hence we may grant the Fredholm property only 
for problem \eqref{eq7}, \eqref{eq9}, \eqref{eq10}, \eqref{eq11}, \eqref{eq.add.prop} 
even under calibration assumption \eqref{eq.calibration}.
However the Fredholm property for a problem is not always the desirable result 
in applications because of the possible lack of the uniqueness and possible absence of 
solutions. As the index (the difference between the 
dimensions of its  kernel and co-kernel) of the Dirichlet problem in the standard setting 
equals to zero, the lack of uniqueness immediately implies some 
necessary solvability conditions applied to the operator in the left 
hand side of \eqref{eq.add.prop}

Moreover, as the coefficient 
$ \frac{\sigma_i \sigma_e \varepsilon \varepsilon_0}{\sigma_e+\sigma_i} $ 
is practically  small, there might be difficulties with numerical 
solving Dirichlet problem  \eqref{eq9}, \eqref{eq10}, 
\eqref{eq.add.prop}. 

Finally, we note that in the practical models of the electrocardiography the term 
$I_{\rm ion} (v,x,t)$ is usually non-linear with respect to $v$. For general non-linear 
Fredholm problems one may provide under reasonable assumptions a discrete set of solutions 
only, see \cite{Sm65} for the second order elliptic operators in H\"older spaces.
Thus one should specify the type of the non-linearities under the consideration.
For example, in the models of the Cardiology the non-linear term is often taken as 
a polynomial of second or third order with respect to $v$, see, for instance, 
\cite{AP96}, \cite{2}, though these choices do not fully correspond to the 
real processes in the myocardium. 
\end{example}

\section{An evolutionary bidomain model}
\label{s.bidiomain.t}

We recall that the primary equations \eqref{eq.balance.charge}, \eqref{eq.balance.current.i}, 
\eqref{eq.balance.current.e}, \eqref{eq.balance.cable.0}, leading to the steady 
bidomain model are actually evolutionary. That is why in this section 
we consider an evolutionary version of the bidomain model adding the  
time variable $t\in [0,T]$, 
\begin{align}
\Delta_i u_i + \Delta_e  u_e  = 0 \mbox{ in } \Omega_m \times [0,T], \label{eq7.t} \\
\Delta_b  u_b  = 0\mbox{ in } \Omega_b \times [0,T],\label{eq8.t} \\
	u_e  = u_b   \mbox{ on } \partial \Omega_m \times [0,T], \label{eq9.t} \\
	\nu_i \cdot (M_e \nabla u_e ) 
	= - \nu_e \cdot (M_b \nabla u_b)\mbox{ on }  \partial \Omega_m  \times [0,T],
	\label{eq10.t} \\
	\nu_i \cdot (M_i \nabla u_i)  =0 \mbox{ on } \Omega_m \times [0,T],  \label{eq11.t} \\
	\nu_e \cdot (M_b \nabla u_b)  = 0 \mbox{ on } \Omega \times [0,T],  \label{eq12.t} \\
		u_b = f \mbox{ on }  \partial \Omega \times [0,T], \label{eq6.t}\\
\label{eq.balance.cable}
\frac{1}{2\chi}\Big(-\Delta_i u _i + \Delta_e u _e\Big) = 
 C_m \frac{\partial (u_i-u_e)}{ \partial t} + I_{\rm ion} 
\mbox{ in } \Omega_m \times (0,T),
\end{align}
with a given function $f (x,t)$ and supplements it with an evolutionary equation in the 
cylinder domain $\Omega_m \times (0,T)$. 

As before, it is reasonable to supplement the model 
with a modified calibration assumption: there is a function 
$c_0 (t) \in C[0,T]$ such that 
\begin{equation} \label{eq.calibration.t}
\int_{\partial \Omega_m} (u_i (y,t) + c_0 (t) u_e (y,t)) 
d\sigma (y) =0 \mbox{ for almost all }  
t\in [0,T].
\end{equation}

\begin{problem} \label{pr.inverse.inside.t}
Given the value $f\in L^2 ([0,T],H^{3/2} (\partial \Omega_m))$ of electrical 
potentials on the boundary of the body domain, find, if possible, 
intracellular potential $u_i \in H^{2,1} (\Omega_m \times (0,T))$, 
extracellular potential $u_e\in H^{2,1} (\Omega_m \times (0,T))$  
and extracardiac potential $u_b \in H^{2,1} (\Omega_b \times (0,T))$ 
satisfying equations \eqref{eq7.t}, \eqref{eq8.t}, \eqref{eq.balance.cable}  and  
boundary conditions \eqref{eq9.t}-\eqref{eq6.t}.
\end{problem}

The further developments essentially depend on the structure 
of the current $I_{\rm ion}$.  We continue the discussion with 
the simple linear case considered in Example \ref{ex.proportional}.

\begin{theorem} \label{t.bidomain.null.t.prop}
Let the coefficients $M_i$, $M_e$, $M_b$ and $a_j$, $0\leq l \leq n$, 
be real analytic over ${\mathbb R}^n$ and bounded at the infinity. Let also 
$\Delta_e = \lambda \Delta_i $ with some $\lambda>0$ and 
\eqref{eq.calibration.t} hold true. 
If  \begin{equation} \label{eq.Iion.lin}
I_{\rm ion} (v) = \sum_{j=1}^n a_j (x)\partial_j v + a_0 (x)v + g  
\end{equation} 
with some function  $g \in L^2 (\Omega_T)$ and some constants 
$a_j$, $0\leq j \leq n$, then Problem \ref{pr.inverse.inside.t} has no more than one solution 
$(u_i, u_e, u_b)$ in the space 
$$
H^{2,1} (\Omega_m \times (0,T))  \times H^{2,1} (\Omega_m \times (0,T)) \times 
 H^{2,1} (\Omega_b \times (0,T)).
$$   
\end{theorem}

\begin{proof} Fix $f \in L^2 ([0,T],H^{3/2} (\partial \Omega))$ admitting a solution 
$u_b \in H^{2,1} (\Omega_b \times (0,T))$ to \eqref{eq8.t}, \eqref{eq12.t}, \eqref{eq6.t}. 
Let $(\hat u_i, \hat u_e, \hat u_b)$ and $(\tilde u_i, \tilde u_e, \tilde u_b)$  
be two solutions to Problem \ref{pr.inverse.inside.t}. Then
the vector 
$(w_i, w_e,w_b) = (\hat u_i, \hat u_e, \hat u_b) -  
(\tilde u_i, \tilde u_e, \tilde u_b) $ 
satisfies
\begin{align}
\Delta_i w_i + \Delta_e  w_e  = 0 \mbox{ in } \Omega_m \times [0,T], \label{eq7.t.null} \\
\Delta_b  w_b  = 0\mbox{ in } \Omega_b \times [0,T],\label{eq8.t.null} \\
	w_e  = w_b   \mbox{ on } \partial \Omega_m \times [0,T], \label{eq9.t.null} \\
	\nu_i \cdot (M_e \nabla w_e ) 
	= - \nu_e \cdot (M_b \nabla w_b)\mbox{ on }  \partial \Omega_m  \times [0,T],
	\label{eq10.t.null} \\
	\nu_i \cdot (M_i \nabla w_i)  =0 \mbox{ on } \Omega_m \times [0,T],  \label{eq11.t.null} \\
	\nu_e \cdot (M_b \nabla w_b)  = 0 \mbox{ on } \Omega \times [0,T],  \label{eq12.t.null} \\
		w_b = 0
		\mbox{ on }  \partial \Omega \times [0,T], \label{eq6.t.null}\\
		\frac{1}{2\chi}\Big(-\Delta_i w _i + \Delta_e w _e\Big) = 
 C_m \frac{\partial (w_i-w_e)}{ \partial t} + I_{\rm ion} (\hat u_i-\hat u_e) - 
I_{\rm ion} (\tilde u_i -\tilde u_e ) ,
\label{eq.balance.cable.0.null}
\end{align}
the last equation being satisfied in $\Omega \times (0,T)$.
Then by Proposition \ref{p.bidomain.null} we have
\begin{equation} \label{eq.nullspace.inside.t}
	\left\{ 
	\begin{array}{ccccc}
	w_b (x,t) & = & 0 & {\rm if} & (x,t)\in \Omega_b \times [0,T],\\ 
	w_e (x,t) & = & w & {\rm if} & (x,t)\in \Omega_m \times [0,T],\\
	 w_i  (x,t)& = & {\mathcal N}_i (-\Delta_e w (\cdot, t),0)(x) & {\rm if} & 
	(x,t)\in \Omega_m \times [0,T],\\ 
\end{array}
\right.
\end{equation}
where ${\mathcal N}_i $ is the Neumann operator related to $\Delta_i$ 
and $w $ is a function from the space $L^2 ([0,T], H^2_0 (\Omega_m)) \cap 
H^{2,1} (\Omega_m \times (0,T)) $ providing that cable equation
\eqref{eq.balance.cable.0.null} is fulfilled and calibration assumption 
\eqref{eq.calibration.t} holds true. 

Since $\Delta_e = \lambda \Delta_i $ with some $\lambda>0$, 
then, according to \eqref{eq.nullspace.inside.prop} and \eqref{eq.nullspace.inside.t}, 
we have
\begin{equation} \label{eq.nullspace.inside.t.prop}
	\left\{ 
	\begin{array}{ccccc}
	w_b (x,t) & = & 0 & {\rm if} & (x,t)\in \Omega_b \times [0,T],\\ 
	w_e (x,t) & = & w & {\rm if} & (x,t)\in \Omega_m \times [0,T],\\
	 w_i  (x,t)& = & -\lambda w & {\rm if} & 
	(x,t)\in \Omega_m \times [0,T],\\
\end{array}
\right.
\end{equation}
where $w $ is a function from 
$L^2 ([0,T], H^2_0 (\Omega_m)) \cap H^{2,1} (\Omega_m \times (0,T))$ 
satisfying the following reduced version of cable equation
\eqref{eq.balance.cable.0}: 
\begin{equation}
\label{eq.balance.cable.0.null.r}
 \chi C_m (\lambda+1) \frac{\partial w}{ \partial t} + 
  \Delta_e w = 
\chi \Big(I_{\rm ion} (\hat u_i-\hat u_e) - 
I_{\rm ion} (\tilde u_i -\tilde u_e ) \Big) 
\mbox{ in } \Omega_m \times (0,T).
\end{equation}
Clearly, 
\begin{equation*} 
\hat v - \tilde v= 
(\hat u_i -\hat u_e)-(\tilde u_i -\tilde u_e ) = w_i-w_e = - (\lambda+1)w, 
\end{equation*}
and then \eqref{eq.Iion.lin}, \eqref{eq.balance.cable.0.null.r} imply
\begin{equation*}
\left\{ 
\begin{array}{lll}
\frac{\partial w}{ \partial t} +
[\chi C_m (\lambda+1)]^{-1} \Delta_e w + C_m^{-1} \Big(
\sum_{j=1}^n a_j \partial_j w + a_0 w \Big) =0 
\mbox{ in } \Omega_m \times (0,T),\\
	w  = 0   \mbox{ on } \partial \Omega_m \times [0,T], \\
	\nu_i \cdot (M_e \nabla w_e ) 
	=0\mbox{ on }  \partial \Omega_m  \times [0,T].\\
	\end{array}
\right.
\end{equation*}

Under the hypothesis of the theorem the parabolic differential operator 
\begin{equation} \label{eq.L}
{\mathcal L}=
 \frac{\partial }{ \partial t} + 
[\chi C_m (\lambda+1)]^{-1}  \Delta_e  + 
C_m^{-1} \Big( \sum_{j=1}^n a_j \partial_j  + a_0 \Big) 
\end{equation}
admits a fundamental solution $\Psi_{\mathcal L}$ and 
hence  the following 
so-called Green formula for the  parabolic operator ${\mathcal L}$ holds true.

\begin{lemma} \label{l.Green.heat} 
Assume that the parabolic differential operator ${\mathcal L}$ admits a fundamental solution 
$\Psi_{\mathcal L}$. Then for all $ T>0$ and all  $u \in H^{2,1} (\Omega_{T})$
the following formula holds:
\begin{equation}  \label{eq.Green.heat}
\left.
\begin{aligned}
u(x, t),\ (x,t)\in \Omega_{T}  \\
0,\ (x, t)\not\in \overline{\Omega_{T}} 
\end{aligned}
\right\} \! = \Big( 
I_{\Omega} (u)   + G_{\Omega} ({\mathcal L}u)  + 
V _{\partial \Omega} \left( 
\partial _{\nu,M}
u \right)  +  W_{\partial \Omega} (u) \Big) (x,t)  .
\end{equation}
\end{lemma}

\begin{proof}  See, for instance, \cite[Ch. 6, \S 12]{svesh}. or \cite[Theorem 2.4.8]{Tark95a} 
for more general linear operators admitting fundamental solutions.
\end{proof}

Taking into account Green formula \eqref{eq.Green.heat}, 
and the fact that $w\in L^2 ([0,T], H^2_0 (\Omega_m))
\cap H^{2,1} (\Omega_m \times (0,T))$ we conclude that 
\begin{equation} \label{eq.Green.heat.w}
\left.
\begin{aligned}
w(x, t),\ (x,t)\in \Omega _m\times (0,T)  \\
0,\ (x, t)\not\in \overline{\Omega_m} \times [0,T] 
\end{aligned}
\right\} \! =  
I_{\Omega_m} (w)(x,t)     .
\end{equation}
It is well known that the elliptic differential operator 
$$
[\chi C_m (\lambda+1)]^{-1}  \Delta_e  + 
C_m^{-1} \Big( \sum_{j=1}^n a_j \partial_j  + a_0 \Big) 
$$
can be reduced by a linear change of space variables to a strongly elliptic 
operator $\Delta_{\tilde M}$ with a positive matrix $\tilde M$. Hence 
the parabolic operator $\mathcal L$ can be reduced to 
the related operator ${\mathcal L}_{\tilde M}$. Thus,
taking in account Example \ref{ex.fs.heat} we may conclude that 
the fundamental solution  $\Psi_{\mathcal L} (x,t)$ is real analytic 
with respect to the space variable  $x$ for each $t>0$. 
In particular, this means that the potential $I_{\Omega} (u)(x,t)$ is real 
analytic with respect to $x$ for each $t>0$, too. However, according to 
\eqref{eq.Green.heat.w}, it equals to zero outside $\overline \Omega_T$.
Therefore it is identically zero for each $t>0$ and then $w\equiv 0$ in $\Omega _T$, 
cf. \cite{KuSh}, for the similar uniqueness theorem related to the heat equation
or \cite{PuSh} for more general parabolic operators. 

Finally, we see that $(w_i, w_e,w_b)=0$ 
because of \eqref{eq.nullspace.inside.t.prop}.
\end{proof}

As for the existence of the solution to Problem \ref{pr.inverse.inside.t}, 
formula \eqref{eq.sol.ui.prop} yields for the case of proportional Laplacians 
under calibration condition \eqref{eq.calibration.t}:
\begin{equation*} 
u_i =\lambda \Big(-u_e + {\mathcal N}_i (0, \nu_i \cdot (M_b \nabla u _b)) + c(t)\Big),
\end{equation*}
where 
$$
c(t)=-c_0(t)\Big(\int_{\partial \Omega_m}  d\sigma (y) \Big)^{-1}
\int_{\partial \Omega_m}  u_b (y,t) d\sigma (y) .
$$
Then cable equation \eqref{eq.balance.cable} and  \eqref{eq9.t}, \eqref{eq10.t} 
lead us to the following non-standard Cauchy problem for parabolic operator \eqref{eq.L} 
with boundary conditions on the lateral  side of the cylinder domain $\Omega_T$:
\begin{equation}
 \label{eq.balance.cable.0.prop}
\left\{
\begin{array}{lll} 
{\mathcal L} u_e = F  &
\mbox{ in } & \Omega_m \times (0,T) ,\\
	u_e  = u_b  & \mbox{ on } & \partial \Omega_m \times [0,T], 	\\
	\nu_i \cdot (M_e \nabla u_e ) 
	= - \nu_e \cdot (M_b \nabla u_b) & \mbox{ on }  & \partial \Omega_m  \times [0,T],
\end{array}
\right.
\end{equation}
where 
\begin{equation*}
F=h (x,t)+\lambda \Big(\frac{\partial}{\partial t} + 
C_m^{-1} \Big( \sum_{j=1}^n a_j \partial_j  + a
_0 \Big) \Big) \Big(
{\mathcal N}_i (0, \nu_i \cdot (M_b \nabla u _b(\cdot,t)) )(x) +c(t)\Big)=
\end{equation*}
\begin{equation*}
h (x,t)+ \lambda \, {\mathcal L} \Big(
{\mathcal N}_i (0, \nu_i \cdot (M_b \nabla u _b (\cdot,t)) ) (x) + c(t)\Big)
\mbox{ in } \Omega_m \times (0,T) 
\end{equation*} 
because 
$$
\Delta_e c(t) = 0 \mbox{ in } \Omega_m \times (0,T) ,
$$
$$
\Delta_e {\mathcal N}_i (0, \nu_i \cdot (M_b \nabla u _b) ) = \lambda
\Delta_i {\mathcal N}_i (0, \nu_i \cdot (M_b \nabla u _b) ) =
0 \mbox{ in } \Omega_m \times (0,T) .
$$
Actually this problem might be ill-posed, see \cite{KuSh}, \cite{PuSh}. 
According to Theorem \ref{t.Cauchy.heat}, we have to check that 
the potential 
$$
\Big(G_{\Omega} (F) + V_{\partial \Omega_m} (- \nu_e \cdot (M_b \nabla u_b)) + 
 W_{\partial \Omega_m} ( u_b) \Big)^+
$$
from $\Omega_b $ to $\Omega$ as a solution $\mathcal F$ to the equation 
$$
{\mathcal L} \mathcal F =0 \mbox{ in }  \Omega \times (0,T).
$$

By Lemma \ref{l.Green.heat}, for $(x,t) \not \in \overline {\Omega_m}
\times [0,T]$ we have  
$$
G_{\Omega_m} (F) (x,t)= W_{\partial \Omega_m}\Big(
{\mathcal N}_i (0, \nu_i \cdot (M_b \nabla u _b (\cdot,\cdot)) )  + c(\cdot)\Big) (x,t)
\Big)  - 
$$
$$
\lambda \Big(
I_{\Omega_m}\Big(
{\mathcal N}_i (0, \nu_i \cdot (M_b \nabla u _b (\cdot,0)) )(x,t)  + c(0) \Big) + 
G_{\Omega_m} (h) (x,t).
$$

Finally, as in \S \ref{s.bidomain}, we note that in the practical models the term 
$I_{\rm ion} (v,x,t)$ is usually non-linear with respect to $v$. Thus, a 
uniqueness/existence theorems for Problem \ref{pr.inverse.inside.t} are closely related  to 
the uniqueness/existence theorems of solutions 
to a non-linear non-standard Cauchy problem for quasilinear parabolic equation 
that is similar to \eqref{eq.balance.cable.0.prop} but with a  non-linear term $F = F(u_e)$.

\section{Numerical results}
\label{s.numerical}

In this section we present numerical results to illustrate some mathematical approaches 
proposed in this paper, namely,  the methods for reconstruction of  transmembrane potentials 
on the myocardial surface of the cardiac chambers. The objectives of this 
study were: 1) to estimate the accuracy of reconstruction of the transmembrane potentials by 
the extracellular electrical potentials on the cardiac surface under assumptions of  
isotropic electrical conductivity of the extracellular, intracellular and extracardiac media; 
2) to estimate the accuracy of reconstruction of the transmembrane potentials by the electrical potentials measured on the human body surface under the same assumptions.

For this propose we performed numerical simulation of electrical activity of the ventricles 
of the human heart.  We used a methodology of cardiac modelling that was described in detail 
in \cite{Ushenin2021}.  Briefly, to obtain a realistic geometry of the torso and heart we 
utilized computed tomography (CT) data of a patient with the structurally normal heart. The 
CT data were taken from a dataset of work \cite{Ushenin2021}.  After segmentation of the 
heart ventricles and torso we generated a high resolution 3D tetrahedral mesh for the final element (FEM) computations.

To simulate cardiac electrical activity, we used the bidomain model 
(\eqref{eq7.t}-\eqref{eq12.t}, \eqref{eq.balance.cable}). We assigned the membrane 
capacitance $C_m= 1$  $\mu F/cm^2$ and the surface-to-volume ratio  $\chi= 400$  $cm^{-1}$.
We assume the torso to be an isotropic volume conductor with a scalar conductivity 
coefficient $m_b=7$ mS/cm and the myocardium to be an anisotropic volume conductor. 
Following \cite {Boulakia2010}, \cite{Clayton2010}, \cite{Keller2012}, the electrical 
conductivity tenzors were constructed as follows:
 $M_i = R\left( \begin{smallmatrix} \sigma_{li}&0&0\\ 0&\sigma_{ti}&0 \\ 0&0&\sigma_{ti} 
\end{smallmatrix} \right)R^{T}$ and $M_e = R\left( \begin{smallmatrix} \sigma_{le}&0&0\\ 0&
\sigma_{te}&0 \\ 0&0&\sigma_{te} \end{smallmatrix} \right)R^{T}$, 
where $\sigma_{li}$, $\sigma_{ti}$ are the intracellular conductivities in the longitudinal 
and transversal direction, $\sigma_{le}$, $\sigma_{te}$ are the extracellular}
conductivities in the longitudinal and transversal direction, $R$ is a matrix called a rotation basis. 

The rotation basis $R$ was defined according to the myocardial fiber orientations, which were 
determined in the myocardium volume by a rule-based approach (see \cite{Bayer2012} for 
details).  We used the following values of the conductivities: $\sigma_{li}=12,\ 
\sigma_{ti}=1.33,\ \sigma_{le}=45,\ \sigma_{te}=5$  mS/cm. 
These conductivities were chosen to provide physiological values of the conduction velocity 
along the myocardial fibers (0.5-0.6 m/s) and  across ones (0.15-0.25 m/s) as well as 
realistic QRS magnitude and duration respect to the QRS properties of the real patient electrocardiogram in 
case of ectopic activation from the ventricle apex. 

We employed the TNNP 2006 cellular model for human ventricle cardiomyocytes 
\cite{ten2006alternans} as a basic model 
to compute the transmembrane ionic current $I_{\text{ion}}$.  Transmural and apico-basal 
cellular heterogeneity of the ionic channels properties was introduced in the model equations 
using the approaches proposed in \cite{ten2006alternans} and \cite{Keller2012}. 

The computations were performed with Oxford Cardiac Chaste software \cite{Cooper2020}. The 
time resolution of the simulated electrical signals was 1,000 frames per second.

We simulated three patterns of electrical excitation of the ventricles of the heart which 
were initiated by focal origins (electrical pacing). The origins were placed: 1) in the 
lateral wall of the left ventricle (LV); 2) in the apex of the heart (Apex); 3) in the right 
ventricle outflow tract (RVOT). 

For the next stage of the numerical experiments we created a medium resolution triangular 
mesh of the surfaces of the heart ventricles and torso for the boundary element (BEM) 
computations. The BEM mesh nodes coincided with a subset of the nodes of the FEM mesh on the 
cardiac and body surfaces. 

The values of the transmembrane potential and extracellular potential on the cardiac surface 
as well as electrical potential values on the surface of the torso obtained by the simulation 
were transferred from the FEM mesh nodes to the respective BEM mesh nodes. As a result, for 
all discrete time moments $t_p$  of the cardiocycle  we got vectors $v(x_j)$  of the 
transmebrane potential values and $u_e(x_j)$ of extracellular potential values in the BEM 
mesh nodes  $x_j\in \partial\Omega_m$ on the cardiac surface as well as a vector $u_b(x_i)$ 
of electrical potential values in the BEM mesh nodes  $x_i\in \partial\Omega$ on the body 
surface. The transmembrane potential values  $v(x_j)$, $x_j\in \partial\Omega_m$ were 
considered as a "ground truth" data.

Next, we recalculated transmembrane potential values  $v_r(x_j)$, $x_i\in \partial\Omega_m$ 
by $u_e(x_j)$ on $\partial\Omega_m$ and by $u_b(x_i)$ on $\partial\Omega$ under the 
assumption of isotropic electrical conductivity of the intracellular and extracellular media 
and compared them with the ground truth transmembrane potential values $v(x_j)$  on 
$\partial\Omega_m$.  For this propose conductivity tensors $M_i$, $M_e$ were approximated by 
scalar coefficients $m_i$, $m_e$ respectively. In this "proof of concept" study we used the 
simplest approach, assuming $m_i =\sigma_{li}$ and $m_e =\sigma_{le}$.

In general, we used the collocation version of BEM (see \cite{Kalinin2019})  for the recomputation 
of the transmembrane potential. 

The first evaluation protocol includes recalculation of the transmembrane potential values 
$v_r(x_j)$, $x_j\in \partial\Omega_m$ by extracellular potential value $u_e(x_j)$,  $x_j\in 
\partial\Omega_m$. The  reconstruction was based on formula 
\eqref{eq.sol.ui.prop}.  When the conductivities are isotropic it takes a form:
 $u_i =-\lambda u_e +\lambda {\mathcal N}_i (0,  m_b\cdot\nu_i \cdot \nabla u _b) +c$.

Actually, the reconstruction of the transmembrane potential included the follows steps:

\begin{enumerate}
\item[(a)] computation of normal derivative $\nu_i\cdot \nabla u _b$ of the body electrical 
potential $u _b$ on the myocardial surface $\partial\Omega_m$;  
\item[(b)] computation of intracellular potential $u_i$  on the myocardial surface 
$\partial\Omega_m$ as $u_i =-\lambda u_e +\lambda {\mathcal N}_i (0, m_b \cdot \nu_i \cdot 
\nabla u _b) +c$;
\item[(c)] calculation of transmembrane potential $v=u_i-u_e$ on $\partial \Omega_m$.
\end{enumerate}

Note that $\nu_i\cdot \nabla u _b = -\nu_e\cdot \nabla u _b$.  Taking in account \eqref{eq8}, 
\eqref{eq9} and \eqref{eq9} $\nu_e\cdot \nabla u _b$ can be computed as the external normal 
derivative of solution $u_b$ to the Zaremba problem for the Laplace equation:

\begin{equation*} 
\left\{ \begin{array}{lll}
\Delta u_b =0 & {\rm in} & \Omega_b,  \\
u _b = u_e   & {\rm on} & \partial \Omega_m,\\   
\partial _{\nu } u = 0  & {\rm on} & \partial \Omega.\\
\end{array}
\right.
\end{equation*}

For this computation we used a BEM approach given in \cite{Kalinin2019} (formula (11)). 

The Neumann-to-Dirichlet transform ${\mathcal N}_i (0, m_b \cdot \nu_i \cdot \nabla u _b)$ can be computed as the trace  on $\partial\Omega_m$ of the solution to the Neumann problem for the Laplace equation:

\begin{equation*} 
\left\{ \begin{array}{lll}
\Delta u_m =0 & {\rm in} & \Omega_m,  \\
\partial _{\nu } u = M_b \cdot \nu_i \cdot \nabla u _b   & {\rm on} & \partial \Omega_m.\\
\end{array}
\right.
\end{equation*}

We used a BEM realization of the Neumann-to-Dirichlet transform given in \cite{RS2007} 
(formula (1.26)); the calibration constant was defined by formula 
\eqref{eq.const}.

The second evaluation protocol includes recalculation of the transmembrane potential values $v_r(x_j)$, $x_j\in \partial\Omega_m$ by electrical potential value $u_b(x_i)$,  $x_i\in \partial\Omega$.

It consists of two steps: 

a) computation of $u _b = u _e$ on  $\partial\Omega_m$ by solving the Cauchy problem for the Laplace equation 

\begin{equation*} 
\left\{ \begin{array}{lll}
\Delta u_b =0 & {\rm in} & \Omega_b,  \\
u _b = f   & {\rm on} & \partial \Omega,\\   
\partial _{\nu } u = 0  & {\rm on} & \partial \Omega.\\
\end{array}
\right.
\end{equation*}

 b) subsequent computations according to the the first evaluation protocol. 

For solving the Cauchy problem we implemented a method similar to one provided by Theorem 
\ref{t.Cauchy.M}. More precisely, we used its BEM realization 
(including the Tikhonov regularization) from paper \cite{Kalinin2019} (formulas (35)-(36)). 
The regularization parameter for the Tikhonov method was obtained by the 
$L$-curve approach. 

To compare the reconstructed transmembrane potentials with the ground truth ones we 
calculated the root mean square error $\delta (v_{j,t}, \hat{v}_{j,t})$:

$\delta = \sqrt{\frac{1}{NT} \sum_{p=1}^{T}\sum_{j=1}^{N} (v_r(x_j, t_p) - v(x_j, t_p) )^2}$,

where  $N$ is the number of the mesh nodes on the cardiac surface, $T$ is the number of  
discrete time points of the cardiac cycle, $v_r$ is the reconstructed transmembrane 
potential,  $v$ is the ground truth transmembrane potential. 

Results of the first evaluation protocol are presented in 
Table 1.  Figure 1 display 
distributions of the reference and the reconstructed trancmembrane potential of the ventricle surface at three consecutive time moments of their depolarization. The comparison of the transmenbrane potential signals in the selected point on the cardiac surface is shown on Figure 2.

\begin{table}[]
	\caption{Root mean square error $\delta$  between the reconstructed transmembrane potential and the ground truth one.}\label{tbl:num_results}
	\begin{tabular}{r|cc}
		 & {$u_e \rightarrow v$} & {$u_b \rightarrow u_e \rightarrow v$} \\ \hline
		LV APEX              & 12.56 mV                                   & 13.84  mV                                                 \\
		LV                      & 5.71 mV                                     & 19.81 mV                                                  \\
		RVOT                 & 13.71 mV                                    & 16.89  mV                                                
	\end{tabular}
\end{table}

\begin{figure}[t] 
	\begin{center}
		\includegraphics[width=0.98\textwidth]{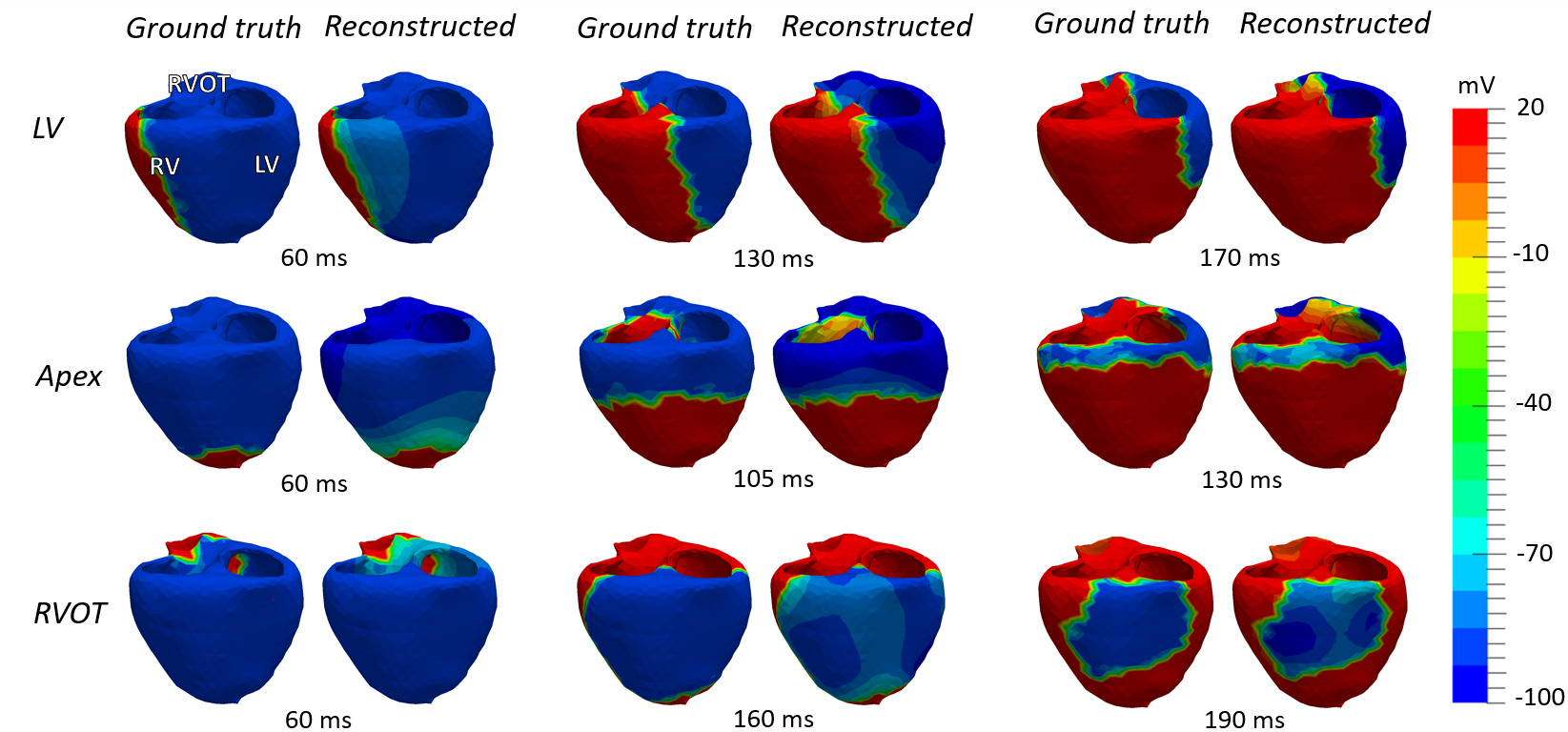}
		\caption{The ground truth transmembrane potential and the reconstructed transmembrane 
		potentials recomputed by the extracellular potential under the isotropic 
		assumptions. LV, 
		the left ventricle; RV, the right ventricle; RVOT the right ventricle outflow tract.}
		\label{fig:img1}
	\end{center}
\end{figure}

\begin{figure}[t] 
	\begin{center}
		\includegraphics[width=0.98\textwidth]{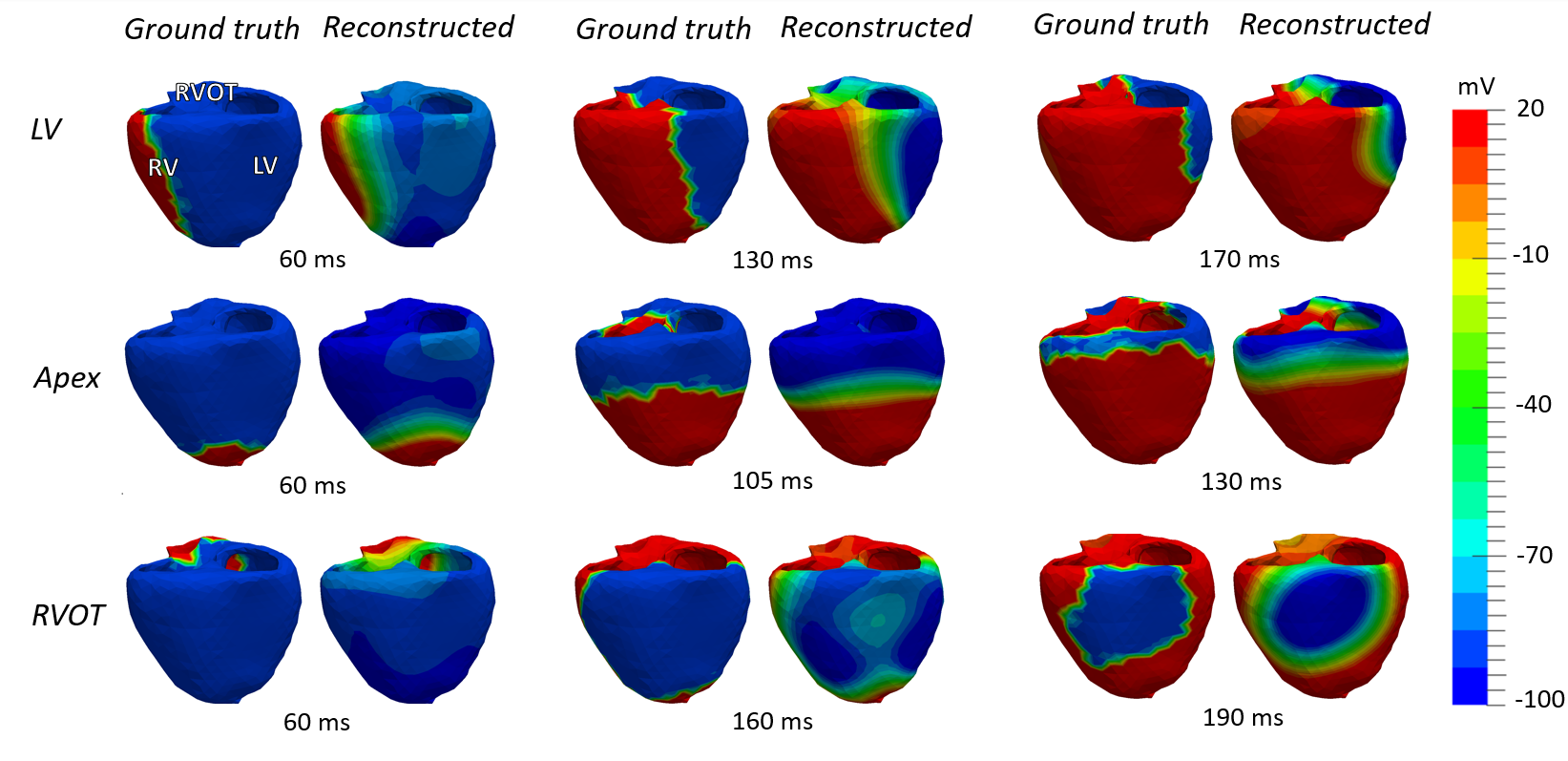}
		\caption{The ground truth transmembrane potential and the reconstructed transmembrane 
		potentials recomputed by the electrical potential on the body surface under the isotropic 
		assumptions. LV, the left ventricle; 
		RV, the right ventricle; RVOT the right ventricle 
		outflow tract.}\label{fig:img2}
	\end{center}
\end{figure}

\begin{figure}[t] 
	\begin{center}
		\includegraphics[width=0.98\textwidth]{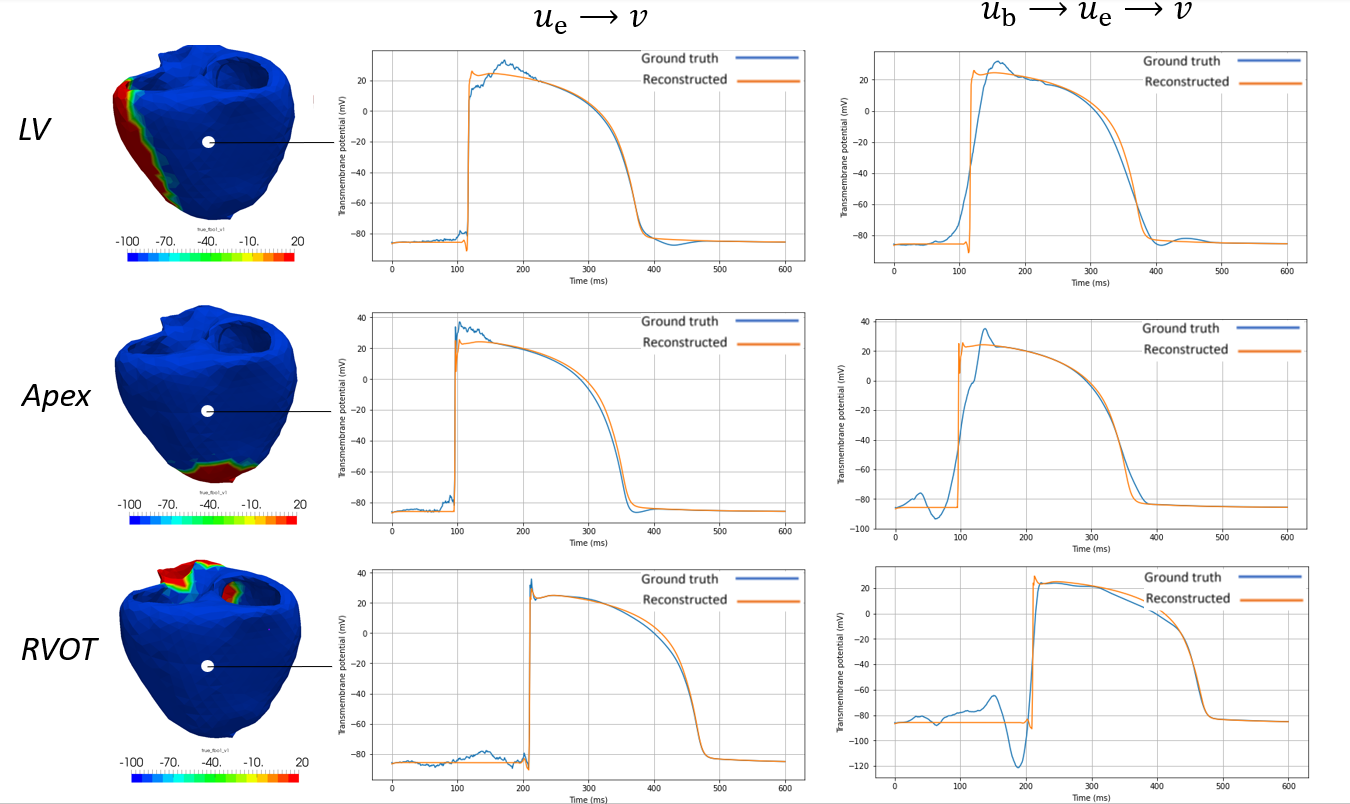}
		\caption{Examples of the ground truth and reconstructed transmembrane potential signals in 
		a point in the center of the posterior-lower region of the ventricles of the heart. }
		\label{fig:img3}
	\end{center}
\end{figure}

This results shows the possibility of a sufficiently accurate reconstruction of the 
transmembrane potential based on the extracellular potential on the cardiac surface under the 
assumption of isotropic intracellular and extracellular electrical conductivity. The maximum 
reconstruction errors were observed in the vicinity of the spike of the transmembrane 
potential signal, while the up-stroke of the signal was reconstructed with high accuracy. 
Probably the precision of the reconstruction can be improved by using more optimal values for 
$m_i$ and $m_e$. 

Results of the second evaluation protocol are presented in Tble 1 and shown in Figures 2 and 3.  As expected, the accuracy of the reconstruction of the transmembrane 
potential was less than in the previous case. At the same time,  the reconstructed 
transmebrane potential correctly conveys the sequence of the myocardial activation and the 
basic shape of the exact transmembrane potential signals.

The main component of the solution distorsion was the smoothness of the activation front and, 
accordingly, the upstroke of the transmembrane signals. Such pattern of the reconstructed 
solution is typical for the Tikhonov regularization. This fact suggests that the application 
of more advanced regularization algorithms. For investigation of the regularisation methods  
a theory of bases with double orthogonality in the Cauchy problem for elliptic operators  
(see \cite{ShTaLMS}) can be useful.

\section{Discussion and conclusion}
\label{s.discussion}

Currently, methods for computational reconstruction of electrical activity of the heart 
inside the myocardium are being intensively developed based on the numerical solution of 
inverse problems for the bidomain model in various statements. This raises an important 
question about the theoretical limit of researchers' endeavors in this direction. In 
particular, the established uniqueness theorems for the inverse problems are very important 
because it provides the basis for numerical computations. In contrast to the “forward” 
initial-boundary value problem for the bidomain equations the uniqueness of the solution of 
inverse problems has not been sufficiently studied. In this work, we aimed to eliminate this 
gap and provide some mathematical background for both the facts that are well adopted 
in the engineering community and some new ideas providing a substantial progress in 
computations.
 
The non-uniqueness of the solution of the inverse problem of reconstruction of the 
transmembrane potential inside the myocardium for the second-order elliptic equation of the 
bidomain model were shown in several previous works \cite {Nielsen2007},\cite{1},\cite{2K}. 
In this paper, we generalized these results by presenting a complete description of the 
null-space of the problem for the case of anisotropic electrical conductivity (Proposition 
\ref{p.bidomain.null}). As a consequence, we also showed the uniqueness of the reconstruction 
of the action potential on the surface of the myocardium 
from the known electrical potential on the surface of the body.

Note that the electrical 
activity of the heart, even on its surface, provides valuable electrophysiological 
information about the patterns of cardiac excitation and the mechanisms 
of cardiac arrhythmias. In contrast to the electrical potential, the transmembrane potential 
more accurately characterizes the local electrical activity of the myocardium, especially the 
processes of myocardial repolarization. Therefore, the reconstruction of the transmembrane 
potential on the surface of the heart, the feasibility of which was justified in this 
article, can be useful for medical applications.

We illustrated the method for reconstruction of the transmembrane potential on the myocardial 
surface by the numerical experiments using the data of personalised modelling of electrical 
activity of the human heart ventricles. The reconstruction method were robust with respect to 
the model error associated with the "isotropic" approximation of tensors of the 
extracellular and intracellular electrical conductivity.

From mathematical point of view, Proposition \ref{p.bidomain.null} states that the steady part
of the bidomain model has too many degrees of freedom. This means that some
of the necessary information about the desired solution is missing. Some 
approaches to complete this information were proposed in \cite {Nielsen2007},\cite{1}, 
\cite{Ainseba2015}. 

In this paper we considered two other possible ways to ensure the uniqueness of
the solution of the problem. The first way consists of introducing the additional
fourth order strongly elliptic equation. We gave an example to show a fact that the forth 
order elliptic equation can be obtained by applying the continuity equation in 
electromagnetism to the bidomain equations. The second way is to consider the original 
evolutionary form of the bidomain model.

The consideration were performed under very restrictive assumptions. Namely, we used the "monodomain"{} assumption about of the proportionality of the electrical conductivity tensors and we utilized a linear version of the activation function of the bidomain model.

These simplifications are significant limitations 
of this study. However, some positive results on the uniqueness of the solution obtained for 
this highly simplified model show the prospects for further research in this direction. 

\smallskip

\textit{Acknowledgments\,} 
The second author was supported by the Russian Science Foundation,  grant N 20-11-20117.

\end{document}